\documentclass[a4paper,reqno,11pt,oneside]{amsart}

\usepackage{amsmath,geometry}
\usepackage{amssymb,amsmath, amsfonts, amsthm}
\usepackage{latexsym}
\usepackage{cases}
\usepackage{tikz, xcolor}
\usepackage[colorlinks]{hyperref}

\makeatletter

\newtheorem{theorem}{Theorem}[section]

\newtheorem{proposition}[theorem]{Proposition}

\theoremstyle{definition}

\newtheoremstyle{Remark}{0.5}{0.5}{\normalfont}{}{%
  \bfseries}{.}{5\p@ plus\p@ minus\p@}{}
\theoremstyle{Remark}
\newtheorem{remark}[theorem]{Remark}
\newtheorem{example}[theorem]{Example}

\numberwithin{equation}{section}

\@ifundefined{leqslant}{}{\let\le=\leqslant  }
\@ifundefined{geqslant}{}{\let\ge=\geqslant  }
\renewcommand{\(}{\left(}
\renewcommand{\)}{\right)}
\newcommand{\R}{{\mathbb R}}
\newcommand{\intd}{\,\mathrm{d}}

\newcommand{\order}{{\mathfrak o}} 
\newcommand{\I}{\mathrm{i}} 
\newcommand{\bvphi}{\boldsymbol{\varphi}}
\newcommand{\abs}[1]{\mathopen| #1\mathclose|}

\DeclareMathOperator{\spanned}{span}
\DeclareMathOperator{\Div}{div}
\newcommand{\eps}{\varepsilon}

\definecolor{n1}{RGB}{220,76,76}
\definecolor{n2}{RGB}{220,157,76}
\definecolor{n3}{RGB}{94,172,59}
\definecolor{n4}{RGB}{116,76,220}
\definecolor{n5}{named}{n1}
\definecolor{n6}{named}{n2}
\definecolor{n7}{named}{n3}
\definecolor{n8}{named}{n4}
\definecolor{n-1}{RGB}{64,85,255}
\definecolor{n-2}{RGB}{76,186,220}

\IfFileExists{graphs.tex}{

  \newcommand{\gr}[1]{##1}
}{
  \newcommand{\gr}[1]{graphs/##1}
}

\makeatother

\title[Normalized vector solutions of nonlinear Schr\"odinger systems]
{Normalized vector solutions of nonlinear Schr\"odinger systems}

\begin{document}

\author{Xiaomeng Huang}
\address[Xiaomeng Huang]{School of Mathematics and Statistics, Hubei University of Education, Wuhan 430205, China}
\email{hhuangxiaomeng@126.com}

\author{Angela Pistoia}
\address[Angela Pistoia]{Dipartimento SBAI, Sapienza Universit\`a di Roma,
via Antonio Scarpa 16, 00161 Roma, Italy }
\email{angela.pistoia@uniroma1.it}

 \author{Christophe Troestler}
\address[Christophe Troestler]{D\'epartement de Math\'ematique,
  Universit\'e de Mons, Place du Parc 20, B-7000, Mons, Belgium}
\email{christophe.troestler@umons.ac.be}

\author{Chunhua Wang}
\address[Chunhua Wang]{School of Mathematics and Statistics, Key Lab NAA--MOE, Central China Normal University, Wuhan 430079, China}
\email{chunhuawang@ccnu.edu.cn}

\thanks{A. Pistoia  is partially supported by  the MUR-PRIN-20227HX33Z
  ``Pattern formation in nonlinear phenomena'' and  partially by
  INDAM-GNAMPA project ``Problemi di doppia curvatura su variet\`a a
  bordo e legami con le EDP di tipo ellittico''. C. Wang was supported
  by National Key Research and Development of China (No.~2022YFA1006900)
  and NSFC (No.~12471106).}


\begin{abstract}
  Given $\mu>0$ we look for solutions $ \lambda\in\R$ and $v_1,\dots,v_k\in H^1(\R^N)$ of the system
  \begin{align*}
    \begin{cases}
      \displaystyle
      -\Delta v_i+ \lambda v_i+V_i(x)v_i
      = \sum_{\substack{j=1}}^k\beta_{ij} v_iv_j^2
      &\text{ in } \R^N, \quad   i=1,\dots,k,\\
      \displaystyle
      \int_{\R^N} \left(v_1^2+\dots+v_k^2 \right)\intd x = \mu,
    \end{cases}
  \end{align*}
  where $N=1,2,3,$ $V_i:\mathbb R^N\to \mathbb R $ and
  $\beta_{ij}\in\R$ satisfy $\beta_{ij}=\beta_{ji}$ and
  $\beta_{ii}>0$.
  Under suitable assumptions on  the $\beta_{ij}$'s, given a non-degenerate critical point $\xi_0$ of a suitable linear combination of the  potentials  $V_i$, we build  solutions
whose components  concentrate at $\xi_0$
   as the prescribed global mass $\mu$ is either large (when
  $N=1$) or small (when $N=3$) or it approaches some
  critical threshold (when $N=2$).
\end{abstract}

\maketitle

\section{Introduction}

A problem widely studied in the last decades concerns  the existence of  solutions
$(\lambda_i, v_i)\in \R\times H^{1}(\R^{N})$, $i=1,\dots,k$ of the systems 
\begin{equation}\label{sps0}
  -\Delta v_i+ \lambda_i v_i+V_i(x)v_i
  = \sum_{\substack{j=1}}^k\beta_{ij} v_iv_j^2
  \text{ in } \R^N, \quad   i=1,\dots,k,
\end{equation}
 with  prescribed masses,  namely
\begin{equation}
  \label{sps0-mass}
  \int_{\R^N} v_i^2= \mu_i,\quad i=1,\dots,k.
\end{equation}

Here $N=1,2,3,$   $V_i\in C^0(\R^N)\cap L^\infty(\R^N) $  and  $\beta_{ij}\in\R$ satisfy
$\beta_{ij}=\beta_{ji}$.  We will consider the \emph{focusing} case, i.e.\ $\beta_{ii}>0$.

\medskip

 Solutions  to \eqref{sps0}--\eqref{sps0-mass}
naturally arise in the study of solitary waves to time-dependent
nonlinear Schr\"{o}dinger equations as
\begin{equation}\label{eq:NLSEt}
  \I\partial_t \Phi _i+ \Delta \Phi_i - V_i(x)\Phi_i + \sum_{\substack{j=1}}^k\beta_{ij}\Phi_i|\Phi_j|^2 = 0,
  \qquad x\in\R^N,\ t\in\R,
\end{equation}
which  has application in nonlinear optics and in the study of Bose-Einstein condensates  \cite{esry,malomed}. Solitary wave
solutions to \eqref{eq:NLSEt} are obtained imposing the {ansatz}
$\Phi_i(x,t) = \operatorname{e}^{\I\lambda_i t} v_i(x)$, where the real
constant $\lambda_i$
and the real valued function $v_i$ satisfy Equation~\eqref{sps0}.
Despite the problem having some relevance in physical problems, only a
few existence (or non-existence) results seem to be known.

\medskip

The natural approach to produce solutions to
\eqref{sps0}--\eqref{sps0-mass} consists in
finding critical points of the energy
\begin{equation}\label{j}
  J(v_1,\dots,v_k) := \frac12\sum\limits_{i=1}^k \, \int\limits_{\R^N}
 \( |\nabla v_i|^2+V_i(x)v_i^2\) \intd x
  -\frac1{4}\sum\limits_{i,j=1}^k \,\int\limits_{\R^N}\beta_{ij}v_i^2v_j^2\intd x
\end{equation}
constrained on the product of spheres
\begin{equation}\label{s}
  S := S_{\mu_1}\times\cdots\times S_{\mu_k},
  \text{ with }
  S_{\mu_i} := \biggl\{v\in H^1(\R^N)\ :\
  \int\limits_{\R^N}|v|^2\intd x = \mu_i\biggr\}.
\end{equation}
The Langrange multipliers are nothing but the unknown real numbers
$\lambda_1,\dots,\lambda_k$.\\

The study of existence of solutions to \eqref{sps0}--\eqref{sps0-mass}
strongly depends
on the dimension $N$.  Indeed, when $N=2$, the
scaling
 $u(x)= v(x/t)$
leaves both the ratio
$\int \abs{\nabla v}^2 \intd x / \int \abs{v}^{4} \intd x$ and
the mass invariant, which is why the power $p=3$ when $N=2$ is called
$L^2$-critical.  In the following, we agree that $N=2$ is the \emph{critical}  regime and we say that $N=1$ is the
\emph{subcritical}  regime  and $N=3$ is the  \emph{supercritical} regime.

\medskip

The situation in the case of a single equation (i.e.\ $k=1$) is quite well understood.
A complete review of the available results in this context goes beyond the aim of this paper. We only quote  the pioneering paper by Jeanjean
\cite{MR1430506} where the author studies the autonomous case   (i.e.\ the 
potential   is a constant) using a variational argument, which  have been widely  employed in the successive literature. 
We also quote the recent paper by Pellacci, Pistoia,
Vaira and Verzini in \cite{ppvv}, where the authors tackle the problem using a different point of view. They use the well-known Lyapunov-Schmidt method keeping the
mass  as the natural parameter in the reduction process and build
solutions with large mass in the subcritical regime, with small mass
in the supercritical regime and with mass close to a certain threshold value
  in the critical regime. We also refer the
interested reader to the references therein.

\medskip

In striking contrast, very few papers concern with the existence of normalized solutions to the system. Moreover, most of the known results  only consider the case of
 $2$ equations in the autonomous case. To describe them it is useful
 to introduce the coupling parameter $\beta := \beta_{12}=\beta_{21}.$ The first result is 
due to Nguyen and Wang  in \cite{nguyen-wang}  in 
 dimension $N=1$ in an \emph{attractive} regime, i.e.\ $\beta>0$.  In $1$D the growth of the nonlinearity is
subcritical so that  the functional $J$ in \eqref{j} is bounded from below on the constraint $S$ in \eqref{s} and normalized solutions
 can be obtained by minimization.
  We observe that in higher   dimensions 
the functional is unbounded from below on the constraint when
$\beta$ is positive, and hence their approach cannot be used. 
The supercritical regime (i.e.\ $N=3$) has been firstly studied by
Bartsch, Jeanjean and Soave   in \cite{bjs} who developed an accurate
minimax argument to find a solution for suitable choices of the
parameters in the attractive case (i.e.\ $\beta> 0$). In particular,
a solution to the system \eqref{sps0}--\eqref{sps0-mass} exists
for every sufficiently small or sufficiently large  $\beta$. We also quote some further generalizations   obtained by   Bartsch, Zhong and Zou \cite{bzz}, Bartsch and Jeanjean \cite{bj-2017} and Li and Zou \cite{lizou}. The existence of a solution in the  repulsive case  (i.e.\ $\beta<0$) 
has been established by Bartsch and Soave \cite{bs-2017,bs-2017-err,bs-2019} who   devise a different variational approach,  based upon the introduction of a further  natural constraint.
In the critical regime (i.e.\ $N=2$) the existence of normalized solutions  is a very subtle
 issue, heavily depending on the prescribed masses  as can already be seen in
 the scalar case and it seems largely open. Very recently, Mederski and Szulkin \cite{meschi}  consider  the case of $k\ge2$ equations and  show the existence of multiple solutions
provided that all the parameters $\beta_{ij}$'s are positive and
satisfy a suitable condition. In particular, they prove that  if
$\beta:=\beta_{ij}$ for any $i\not=j$ then the system has a solution
when $\beta$ is  large enough. Finally, as far as we know, there is
only one paper concerning the non-autonomous case. Noris, Tavares and
Verzini in \cite{ntv-dcds}  consider the system
\eqref{sps0}--\eqref{sps0-mass} with only two equations in the
presence of   positive  continuous trapping potentials (i.e.\
$V_i\to +\infty$ as $|x|\to\infty$)  and prove via a variational approach the existence of positive solutions with small masses.

\medskip 

In this paper, we study the system \eqref{sps0} when we prescribe  the \emph{global mass} of the solution $v_1,\dots,v_k$. More precisely, 
given $\mu>0$ we look for solutions $ \lambda\in\R$ and ${\bf v}:=(v_1,\dots,v_k),$ $v_i\in H^2(\R^N)$ of the system
\begin{align}\label{sps}
  \begin{cases}
    \displaystyle
    -\Delta v_i+ \lambda v_i+V_i(x)v_i
    = \sum_{\substack{j=1}}^k\beta_{ij} v_iv_j^2
    &\text{ in } \R^N, \quad   i=1,\dots,k,\\
    \displaystyle
    \int_{\R^N} \(v_1^2+\dots+v_k^2\) \intd x = \mu.
\end{cases}
\end{align}
 For sake of simplicity we will assume that $V_i,|\nabla V_i|\in
L^\infty(\R^N)$ for every $i=1,\dots,k$.

Let us introduce the necessary ingredients to state our result.
Let $U$ be the positive radial solution of
\begin{equation*}
  -\Delta U+U=U^3 \quad\text{in } \R^N.
\end{equation*}
It is well known that $U$ and its first and second derivatives
decay exponentially~\cite{Li-Ni,Li-Zhao}.
We assume that $\mathbf{U} := (U_1,\dots, U_k)$ is a
\emph{synchronized} radial positive solution to the limit system
\begin{align}\label{limsps}
  - \Delta U_i + U_i
  = \sum_{j=1}^k\beta_{ij} U_iU_j^2
  \text{ in } \R^N,\quad i=1,\dots,k,
\end{align}
i.e.\ $U_i=\sigma_iU,$ with $\sigma_i>0$ for $i =1,\dots,k$ solutions
of the algebraic system
\begin{equation}\label{ccc}
  \sum\limits_{j=1}^k\beta_{ij}  \sigma_j^2=1, \quad  i=1,\dots,k.
\end{equation}
We also assume that it is \emph{non-degenerate}, i.e.\ the set of the
solutions of the linear system
\begin{equation}\label{linear-eq}
  -\Delta \phi_i + \phi_i
 -\sum\limits_{j=1}^k \beta _{ij}(U_j^2\phi_i+2U_iU_j\phi_j)
  = 0  \text{ in } \R^N,\quad i=1,\dots,k,
\end{equation}
is a $N-$dimensional space generated by
\begin{equation}\label{zi}
  {\boldsymbol \Phi}_i
  := \left( \frac{\partial U_1}{\partial x_i},\dots,
    \frac{\partial U_k}{\partial x_i} \right)
  = \left(\sigma_1 \frac{\partial U}{\partial x_i},\dots,
    \sigma_k \frac{\partial U}{\partial x_i}\right),
  \quad i=1,\dots,N.
\end{equation}
Examples of this kind of solutions can be found in Examples \ref{ex1}
and \ref{ex2}.

Set
\begin{equation}\label{so}
  \mu_0 := \sum\limits_{i=1}^k\int\limits_{\R^N} U_i^2(x)  \intd x
  = \gamma \sum\limits_{i=1}^k \sigma_i^2
  \quad \text{with } \gamma =\int_{\R^{N}}U ^{2}\intd x
  .
\end{equation}

Next, we introduce the \emph{global potential}  (see \eqref{so})
\begin{align}\label{gama}
  \Gamma(x) = \gamma\sum\limits_{i=1}^k\sigma_i^2V_i(x).
\end{align}
We assume that $\xi_0\in \R^{N}$ is a non-degenerate critical point of $\Gamma$.
Without loss of generality we can  suppose that, in a neighborhood of
$\xi_0$,
\begin{equation}\label{xiogama}
  \Gamma(x)
  = \Gamma(\xi_0)+\frac12\sum\limits_{i=1}^N
  \frac{\partial^2\Gamma}{\partial x_i^2}(\xi_0)(x-\xi_0)_i^2
  + \mathcal O\(|x-\xi_0|^3\),
  \quad\text{with } \frac{\partial^2\Gamma}{\partial x_i^2}(\xi_0)\ne0.
\end{equation}
We will also assume that each single potential $V_i$ is $C^4$ in a neighbourhood of $\xi_0$.

\medskip

 Finally, we say that a family ${\bf v}={\bf v}_\mu$ of solutions of \eqref{sps},
indexed on $\mu$, \emph{concentrates} at 
$\xi_0 \in \R^N$ if
\begin{equation*}
  {\bf v}_\mu (x) = \frac 1{\epsilon_\mu}
  {\bf U} \(\frac{x-\xi_\mu}{\epsilon_\mu}\) + \phi_\mu(x),
\end{equation*}
where, as $\mu\to\mu^*\in[0,+\infty]$, for $\epsilon_\mu\to 0$,
$\xi_\mu \to \xi_0$, and the remainder $\phi_\mu$ is a higher order
term, in some suitable sense.

Finally, we can state our main result.

\begin{theorem}\label{main-ppvv}
 \begin{enumerate}
 \item  There exists $\kappa=\kappa(N)>0$ such that
  \begin{itemize}
  \item[(i)] in the subcritical regime, i.e.\ $N=1$,
    for any $\mu>\kappa$ there exist a solution $(\lambda_\mu, {\bf v}_\mu)$
    to \eqref{sps} with $ {\bf v}_\mu$ concentrating
    at $\xi_0$ 
  \item[(ii)] in the supercritical regime, i.e.\ $N=3$,
    for any $0<\mu<\kappa$ there exist a solution $(\lambda_\mu, {\bf v}_\mu)$
    to \eqref{sps} with $ {\bf v}_\mu$
    concentrating at $\xi_0$ 
  \end{itemize}
  and in both cases
  $$\lambda_\mu\sim \(\frac{\mu_0}\mu\) ^{3\over N-2}\to +\infty\ \hbox{as}\ \mu\to\infty\ \hbox{or}\ \mu\to0,\ \hbox{respectively}.$$
 \item
  In the critical regime, i.e.\ $N=2,$ we suppose that  $V_i(\xi_0)=c$ and $\nabla V_i(\xi_0)=0$ for every $i=1,\dots,k.$
	  Moreover we also assume $\Delta \Gamma (\xi_0)\ne 0$.\\
  Then there exists $\delta >0$ such that for any
  $\mu_0-\delta <\mu< \mu_0$ (if $ \Delta \Gamma(\xi_0)>0$) or
  $\mu_0<\mu< \mu_0 + \delta$ (if $ \Delta \Gamma(\xi_0)<0$) there
  exists a solution $(\lambda_\mu,  {\bf v}_\mu)$  to \eqref{sps} with $ {\bf v}_\mu$
  concentrating at $\xi_0$ and $$\lambda_\mu\sim \( \Delta \Gamma (\xi_0)\over \mu_0-\mu\)^\frac12 \to +\infty\ \hbox{as}\
 \mu\to \mu_0.$$
  \end{enumerate}
\end{theorem}

Let us make some comments.

\begin{remark}
We build the solution using a Ljapunov-Schmidt procedure taking the mass $\mu$ as a parameter in the same spirit of \cite{ppvv}.
The profile of the solution at the main order looks like the synchronized solution to the limit system \eqref{limsps}. 
However, in contrast to the previous work, here the solution must also be corrected at second order by means of the solution of the linear problem \eqref{ans-cor} where the values of the potentials at $\xi_0$ appear.
We observe that in the case of the single equation
  once we fix the non-degenerate critical point $\xi_0$ of $V$ we can
  assume (without loss of generality) that $V(\xi_0)=0$ up to
  replacing $\lambda$ with the new parameter $\lambda-V(\xi_0)$.
This no longer holds in the case of the system, because if  the single parameter $\lambda$
  is replaced by  the parameters $\lambda_i=\lambda-V_i(\xi_0)$   they
  are different, unless all the potentials have the same value at the
  point $\xi_0$.
\end{remark}

\begin{remark}  The main term  of the solution found in Theorem \ref{scexi-1} in the critical regime (i.e.\ $N=2$)
is not good enough to detect its mass. We need to improve the ansatz up to the  second order and to keep the
size of the error term  small enough. That is why we need to assume that all the potentials have the same expansion (up to the first order)
close to the point $\xi_0$, i.e.\ all the functions $V_i$'s take the same value at the point, which also turns out to be a common critical point.
It would be extremely interesting to determine whether this extra assumption is merely a tool to simplify the computations or if it has a deeper significance.
\end{remark}

\begin{remark}
  Byeon in \cite{byeon} considers a singularly perturbed system  with only two equations  similar to system \eqref{sch-1}. He proves the existence of solutions concentrating at the same point which is a common non-degenerate critical point of both the potentials.
In Theorem \ref{scexi-1} we show that the concentration phenomenon is actually governed by the critical points of the \emph{global} potential $\Gamma$ rather than by the critical points of the individual potentials. Moreover, our approach allows to consider  systems with more than two components.
\end{remark}

\begin{remark}
  The existence of solutions concentrating at the point $\xi_0$  
  strongly depends on the nature of the critical point. In particular
  in the critical regime there exists a solution with a mass smaller than
  $\mu_0$ if $\xi_0$ is a minimum point (since $\Delta \Gamma(\xi_0)>0$) or
  with a mass larger than $\mu_0$ if $\xi_0$ is a maximum point (since
  $\Delta \Gamma(\xi_0)<0$). It would be interesting to prove that such
  conditions are also necessary. More precisely, it could be
  challenging to prove that if $\xi_0$ is a minimum or a maximum point
  then there are no solutions blowing-up at $\xi_0$ with masses
  approaching $\mu_0$ from above or below, respectively.
\end{remark}

\begin{remark}
 In  \cite{ppvv} we conjectured that the constant $\alpha_N$ defined in \eqref{alfan}
  is positive for any $N$. This is true in the 1-dimensional
  case as proved in \cite[Remark 3.5]{ppvv}. In Section~\ref{numeric}
  we provide numerical evidence that this is still true for dimensions
  $N\in\{2,\dots,8\}.$ The validity of the conjecture is in our
  opinion an interesting open problem.
\end{remark}

The proof paper is organized as follows.  In Section~\ref{1}, we find a
solution to the perturbed Schr\"{o}dinger equation~\eqref{sch-1} via
the classical Ljapunov-Schmidt reduction. For sake of completeness we
repeat the main steps of the proofs taking into account that a second
order expansion of the main term of the solution we are looking for is
needed. In Section~\ref{3}, we select the solutions with the
prescribed norm. In Section~\ref{numeric} we discuss the
numerical approach used to study the sign of $\alpha_N$ defined in \eqref{alfan}.
Finally, in Appendix~\ref{app} we study the existence of non-degenerate synchronized solutions to System~\eqref{limsps}.

\medskip

\emph{Notation:\/} In what follows we agree that notation
$f=\mathcal O(g)$ or $f\lesssim g$ stand for $|f|\le C|g|$ for some $C >0$ uniformly
with respect all the variables involved, unless specified.

 \section{Existence of   solutions to a   singularly perturbed system}\label{1}
Set
\begin{equation}\label{nui}
  \epsilon:=\lambda^{-\frac{1}{2}}
  \quad\text{and}\quad
  u_i:=\epsilon v_{i},\ i=1,\dots,k.
\end{equation}
Problem~\eqref{sps} turns out to be equivalent to
\begin{numcases}
  \displaystyle
  -\epsilon^2 \Delta u_i + u_i + \epsilon^2 V_i(x) u_i
  = \sum_{j=1}^k\beta_{ij} u_iu_j^2
  &in $\R^N$, \quad $i=1,\dots,k$,
  \label{sch-1}\\
  \displaystyle
  \epsilon^{-2}\int_{\R^N} u_1^2+\dots+u_k^2 \intd x = \mu.
  \label{sps1-mass}
\end{numcases}
It is clear that $(\lambda(\mu), \mathbf{v}(\mu))$,
$\mathbf{v}(\mu):=(v_1(\mu),\dots, v_k(\mu))$ solves \eqref{sps} if and
only if $(\epsilon(\mu),\mathbf{u}(\mu))$,
$\mathbf{u}(\mu) := (u_1(\mu),\dots, u_k(\mu))$ solves
\eqref{sch-1}--\eqref{sps1-mass}.  As a
consequence, the first step is building a solution
$\mathbf{u}=\mathbf{u}(\epsilon)$ to the singularly perturbed
Schr\"{o}dinger system~\eqref{sch-1}
which concentrates at a given point $\xi_0$ as $\epsilon\to0$.

\medskip

In this section, we mainly prove the following result.
\begin{theorem}\label{scexi-1}
  There exists $\epsilon_0>0$ such that for any
  $\epsilon\in(0,\epsilon_0)$ there exists a unique solution
  $\mathbf{u}_\epsilon=( {u}_{1,\epsilon},\dots, {u}_{k,\epsilon})$ to
  \eqref{sch-1} such that
  \begin{equation*}
    u_{i,\epsilon}(x)
    = U_i\(\frac{x-\xi_\epsilon}\epsilon\)
    - \epsilon^2 Z_{i }\(\frac{x-\xi_\epsilon}\epsilon\)
    + \psi_{i,\epsilon} (x),\quad i=1,\dots,k,
  \end{equation*}
  for some $\xi_\epsilon\to \xi_0$ as $\epsilon\to0$.  The functions
  $Z_1,\dots,Z_k \in H^1(\R^N)$ are the radial solutions
  to the linear
  system
  \begin{equation}\label{ans-cor}
    -\Delta Z_i+ Z_i- \sum_{\substack{j=1}}^k \beta _{ij}
    \bigl(U_j^2Z_i+2U_iU_jZ_j \bigr) = V_i(\xi_0)U_i
    \text{ in } \R^N,\quad i=1,\dots,k
  \end{equation}
  and the remainder terms $\psi_{i,\epsilon}$ satisfy
  \begin{equation*}
   \Biggl(\ \int\limits_{\R^N} \epsilon^2|\nabla \psi_{i,\epsilon}|^2
    + \psi_{i,\epsilon }^2 \intd x\Biggr)^{1/2}
    = \mathcal O\Bigl(\epsilon^{\frac N2 + 3}\Bigr).
  \end{equation*}
  Moreover, the map
  $(0,\epsilon_0) \to \bigl(H^1(\R^N)\bigr)^k: \epsilon \mapsto
  \mathbf{u}_\epsilon$ is continuous.
\end{theorem}

\subsection{Preliminaries}
Note that the functions $ {u}_i $ solve system \eqref{sch-1} if and
only if the scaled functions $ {u}_{i}(\epsilon {\,\cdot} +\xi)$, which
will still be denoted by $ u_i$ solve the following system
\begin{align}\label{sps2}
  -\Delta  {u}_i+  u_i+\epsilon^{2}V_i(\epsilon x+\xi)  {u}_i
  = \sum_{\substack{j=1}}^k\beta_{ij} u_iu_j^2
  \text{ in } \R^N,\quad i=1,\dots,k.
\end{align}
Here, we choose the concentration point
\begin{equation*}
  \xi:= \epsilon \tau + \xi_{0},\quad \tau\in\R^N.
\end{equation*}
Let $H^1(\R^N)$ be equipped with the standard scalar product
\begin{equation*}
  \left\langle \phi,\psi\right\rangle
  := \int\limits_{\R^N}\nabla \phi\nabla \psi+ \int\limits_{\R^N} \phi\psi,
\end{equation*}
which induces the standard norm denoted by $\|\cdot\|$. We also denote
by $\mathtt{i}^*:L^{4/3}(\R^N)\to H^1(\R^N)$ the adjoint operator of the
embedding ${\mathtt i}:H^1(\R^N)\hookrightarrow L^{4/3}(\R^N)$, i.e.
\begin{equation*}
  {\mathtt i}^* f=u
  \quad \Longleftrightarrow\quad
  -\Delta u+   u=f\ \text{in}\ \R^N.
\end{equation*}
We  also observe that there exists $C>0$ such that
\begin{equation}\label{norm-es-1}
  \|u\| \le C \|f\|_{L^{4/3}(\R^N)},\quad i=1,\dots,k.
\end{equation}
We set $H:=H^1(\R^N)\times\dots\times H^1(\R^N)$, equipped with the
scalar product and the induced norm (respectively)
\begin{equation}
  \label{eq:scalar-product}
  \langle \mathbf{u},\mathbf{v}\rangle
  := \sum_{i=1}^k\langle u_i,v_i\rangle
  \quad \text{and}\quad
  \|\mathbf{u}\| ^{2}= \sum_{i=1}^k\|u_i\|^{2},
\end{equation}
where $\mathbf{u} = (u_1,\dots,u_k), \mathbf{v} = (v_1,\dots,v_k)\in H$.
Finally, we can rewrite system \eqref{sps2} as
\begin{align}\label{p}
  {u}_i={\mathtt i}^*
  \Bigg\{
  \sum_{\substack{j=1 }}^k\beta_{ij} u_iu_j^2-\epsilon^{2}V_i(\epsilon
  {\,\cdot} +\xi)  {u}_i\Bigg\}
  \text{ in } \R^N, \quad i=1,\dots,k.
\end{align}
We look for a solution to \eqref{p} as
\begin{equation}
  \label{ans}
  \mathbf{u} = \underbrace{\mathbf{U}-\epsilon^2 \mathbf{Z}}_{=:{\mathbf{W}}}
  + \boldsymbol{\phi},
\end{equation}
where $\mathbf{U}:=(U_1,\dots,U_k)$ is the non-degenerate synchronized
solution of the limit system \eqref{limsps}, the correction term
$\mathbf{Z} := (Z_1,\dots,Z_k)$ solves the linear system
\eqref{ans-cor} and the remainder term
${\boldsymbol\phi} := (\phi_1,\dots,\phi_k)\in K^\perp$ where (see
\eqref{zi})
\begin{align}\label{K}
  K:= \spanned\left\{{\boldsymbol\Phi}_i
  := \left(\frac{\partial U_1}{\partial x_i},\dots,
  \frac{\partial U_k}{\partial x_i}\right)
  \ : \ i=1,\dots,N \right\}
\end{align}
and
\begin{align}\label{Kperp}
  K^\perp
  =\bigl\{
  {\boldsymbol \phi}:=(\phi_1,\dots,\phi_k)\in H
  \ :\ \langle  {\boldsymbol \phi},{\boldsymbol\Phi}_i\rangle=0,\
  i=1,\dots,N \bigr\}.
\end{align}
Note that $\mathbf{W} \in K^\perp$ because $U$ and the $Z_i$ are radial.
We rewrite the system  \eqref{p} as follows
\begin{equation}\label{pp}
  {\boldsymbol{\mathcal L}}_{\epsilon,\tau}({\boldsymbol \phi})
  - {\boldsymbol{\mathcal E}}_{\epsilon,\tau}
  - {\boldsymbol{\mathcal N}}_{\epsilon,\tau}({\boldsymbol \phi})
  = 0.
\end{equation}
Here the linear operator ${\boldsymbol{\mathcal L}}_{\epsilon,\tau}$
is defined by
\begin{equation}\label{elle}
  {\mathcal L}_i(\phi_1,\dots,\phi_k)
  := \phi_i-{\mathtt i}^*\left\{
    \sum\limits_{j=1}^k \beta_{ij}
    \bigl( W_j^2\phi_i+2W_iW_j\phi_j \bigr)
    - \epsilon^2V_i(\epsilon{\,\cdot} +\xi)\phi_i
  \right\},
\end{equation}
the nonlinear term
${\boldsymbol{\mathcal N}}_{\epsilon,\tau}({\boldsymbol \phi})$ is
defined by
\begin{equation}\label{enne}
  \begin{aligned}
    &{\mathcal N}_i(\phi_1,\dots,\phi_k)
      ={\mathtt i}^*\Biggl\{ \sum\limits_{j=1}^k\beta_{ij}
      \bigl( W_i\phi_j^2+\phi_i\phi_j^2+2W_j\phi_i\phi_j\bigr)
      \Biggr\}
  \end{aligned}
\end{equation}
and the error term
${\boldsymbol{\mathcal E}} _{\epsilon,\tau}$ is defined by
\begin{equation}\label{ei}
  {\mathcal E}_i
  := {\mathtt i}^* \Biggl\{ \sum\limits_{j=1}^k\beta_{ij}
  W_i W_j^2 - \epsilon^2V_i(\epsilon{\,\cdot} +\xi)W_i
  \Biggr\} - W_i.
\end{equation}
Then, problem \eqref{p} turns out to be equivalent to the system
\begin{align}\label{sc31-1}
  \Pi^\perp\bigl\{
  {\boldsymbol{\mathcal L}}_{\epsilon,\tau}(\boldsymbol{\phi})
  - {\boldsymbol{\mathcal E}}_{\epsilon,\tau}
  - {\boldsymbol{\mathcal N}_{\epsilon,\tau}}(\boldsymbol{\phi})
  \bigr\}
  =0
\end{align}
and
\begin{align}\label{sc32-1}
  \Pi\bigl\{
  {\boldsymbol{\mathcal L}}_{\epsilon,\tau}({\boldsymbol \phi})
  - {\boldsymbol{\mathcal E}}_{\epsilon,\tau}
  - {\boldsymbol{\mathcal N}_{\epsilon,\tau}}({\boldsymbol \phi})
  \bigr\}
  =0,
\end{align}
where $\Pi: H\to K$ and $\Pi^\perp: H \to K^\perp$ are the orthogonal
projections.

\subsection{Solving (\ref{sc31-1})}

\begin{proposition}
  \label{phi}
  For any compact set $T\subset\R^N$ there exists $\epsilon_0>0$ and
  $C>0$ such that for any $\epsilon\in [0,\epsilon_0]$ and for any
  $\tau\in T$ there exists a unique
  $\boldsymbol{\phi} = \boldsymbol{\phi}_{\epsilon,\tau}\in K^\perp$
  in a neighborhood of $0$
  which solves equation
  \eqref{sc31-1} and
\begin{equation}\label{stima-c0}
  \|\boldsymbol\phi_{\epsilon,\tau} \|\le C \epsilon^3 .
\end{equation}
   Moreover, the map
  $  \epsilon  \mapsto
 \boldsymbol \phi_{\epsilon,\tau}$ is continuous and
 the map $  \tau  \mapsto
 \boldsymbol \phi_{\epsilon,\tau}$ is  $C^1$ and satisfies
 \begin{equation}\label{stima-c1}
   \|\partial_\tau\boldsymbol\phi_{\epsilon,\tau} \|\le C \epsilon^{3} .
\end{equation}
\end{proposition}
\begin{proof}
  Let us sketch the main steps of the proof.
\\
(i) First of all, we prove that the linear operator
    ${\boldsymbol{\mathcal L}}_{\epsilon,\tau}$ is uniformly
    invertible in $K^\perp,$ namely there exists $\epsilon_0>0$ and
    $C>0$ such that
    \begin{equation*}
      \|{\Pi^\perp \boldsymbol{\mathcal L}}_{\epsilon,\tau}(\bvphi)\|
      \ge
      C\|\bvphi\|
      \quad \text{for any}\ \epsilon\in [0,\epsilon_0],\ \tau\in T
      \text{ and }
      \bvphi\in K^\perp.
     \end{equation*}
    Observe that
    \begin{equation*}
      \Pi^\perp \boldsymbol{\mathcal L}_{\epsilon,\tau}(\bvphi)
      = \Pi^\perp \boldsymbol{\mathcal L}_0(\bvphi)
      + \epsilon^2
      \widetilde{\boldsymbol{\mathcal L}}_{\epsilon,\tau}(\bvphi),
    \end{equation*}
    where the linear operator
    $\widetilde{\boldsymbol{\mathcal L}}_{\epsilon,\tau}$ is uniformly
    bounded and the linear operator $\boldsymbol{\mathcal L}_0$
    defined by
    \begin{equation}
      \label{eq:def-L0}
      (\boldsymbol{\mathcal L}_{0})_i(\varphi_1,\dots,\varphi_k)
      = \varphi_i-{\mathtt i}^* \Biggl\{\sum_{j=1}^k\beta_{ij}
      \(U_j^2\varphi_i+2U_iU_j\varphi_j\)\Biggr\}.
    \end{equation}
    The non-degeneracy assumption \eqref{linear-eq} means that
    $K = \ker \boldsymbol{\mathcal L}_0$.  Given that
    $\boldsymbol{\mathcal L}_0$ is self-adjoint,
    $\Pi^\perp \boldsymbol{\mathcal L}_0 = \boldsymbol{\mathcal L}_0$
    and, because it is a compact perturbation of the identity,
    it is invertible.

    \medskip

(ii) Next, we compute the size of the error
    ${\boldsymbol{\mathcal E}}_{\epsilon,\tau}$ in terms of
    $\epsilon$.  We observe that by \eqref{limsps}
    \begin{equation*}
      U_i = \mathtt i^*\Biggl\{\sum_{j=1}^k\beta_{ij}U_iU_j^2\Biggr\}
    \end{equation*}
    and by \eqref{ans-cor}
    \begin{equation*}
      Z_i = \mathtt i^*\Biggl\{ \sum_{j=1}^k\beta_{ij}
      \bigl( U_j^2Z_i +2U_iU_jZ_j\bigr)
      + V_i(\xi_0)U_i \Biggr\} .
    \end{equation*}
    Moreover, by the mean value theorem
    (recalling that $|\nabla V_i| \in L^\infty$),
    \begin{equation*}
      V_i(\epsilon x+\epsilon \tau+\xi_0)
      = V_i(\xi_{0}) + \mathcal O\bigl(\epsilon  (1+|x|) \bigr).
    \end{equation*}
    Combining the above facts we have
    \begin{align}\label{errore}
      \mathcal E_i
      \nonumber
      &={\mathtt i}^*\Biggl\{ \epsilon^4 \sum_{j=1}^k\beta_{ij}
        \bigl(U_i Z_j^2 + 2U_jZ_iZ_j\bigr)
        - \epsilon^{6}\sum_{j=1}^{k}\beta_{ij}Z_{i}Z^{2}_{j}\Biggr\}\\
      &\qquad
        +{\mathtt i}^*\bigl\{ \epsilon^4V_i(\epsilon{\,\cdot} +\xi)Z_i
        -\epsilon^2\(V_i(\epsilon{\,\cdot} +\xi)-V_i(\xi_0)\)U_i\bigr\}
    \end{align}
    and
    $$
    \|\mathcal E_i\|
    \le C \epsilon^3, \quad i=1,\dots,k.
    $$

    \medskip

(iii) The existence part follows by  a standard contraction mapping argument.  The
    contraction relies on the inequality
    $\|\boldsymbol{\mathcal N}_{\epsilon,\tau}(\boldsymbol\phi_1) -
    \boldsymbol{\mathcal N}_{\epsilon,\tau}(\boldsymbol\phi_2) \| \le
    c \|\boldsymbol\phi_1 - \boldsymbol\phi_2\| $ where
    $c = \mathcal O\bigl( \|\boldsymbol\phi_1\| +\|\boldsymbol \phi_2\|\bigr)$ which can be
    deduced combining the mean value theorem and \eqref{enne}.  The
    fact that a small ball around $\boldsymbol\phi = 0$ is mapped into
    itself comes from point (ii) and the following estimate
    $\|\boldsymbol{\mathcal N}_{\epsilon,\tau}(\boldsymbol\phi)\| \le
    C \|\boldsymbol\phi\|^2$ (which follows by \eqref{enne}) valid in
    a neighbourhood of $\boldsymbol\phi = 0.$ Point (ii) and
    this last inequality also imply the bounds on
    $\|\boldsymbol\phi_{\epsilon,\tau}\|$.  Finally, the continuity of
    the fix point $\boldsymbol\phi_{\epsilon, \tau}$ follows from the
    same continuity of the contracting map (we choose
      $\epsilon_0$ small enough so that, for all $\epsilon \in
      [0,\epsilon_0]$, $\epsilon T$ lies in the neighborhood of
      $\xi_0$ where all $V_i$'s are of class $C^4$).  See e.g.\
    \cite[Proposition~3.5]{Micheletti-Pistoia-09} for more details.

    \medskip

(iv) We show that the map $\tau\mapsto
  \boldsymbol\phi_{\epsilon,\tau}$ is a $C^1$.  Our arguments are
  inspired by those developed in
     \cite[Proposition~3.5]{Micheletti-Pistoia-09}  (see also   \cite[Proposition~5.2]{dfm}).
     We apply the Implicit Function Theorem to the $C^1$-function
     ${\bf G}:\mathbb R^N\times K^\perp \to K^\perp$ defined by
   \begin{equation*}
       {\bf G}(\tau,\bvphi)
       := \bvphi - \Pi^\perp\left\{
         \mathbf{F}\bigl(\mathbf{W} + \bvphi\bigr) - \mathbf{W} \right\}
   \end{equation*}
     where $ {\mathbf{W}}$ is  defined in \eqref{ans} and
  the function $\mathbf{F}:H\to H$
is defined as (see \eqref{p})
  \begin{equation*}
    F_i(\mathbf{u}
    ) := \mathtt{i}^* \left\{
      \sum_{\substack{j=1}}^k\beta_{ij} u_iu_j^2
      - \epsilon^2  V_i(\epsilon {\,\cdot} +\xi)  u_i \right\},
    \quad i=1,\dots,k.
  \end{equation*}
 Now, it is clear that ${\bf
   G}(\tau,\boldsymbol\phi_{\epsilon,\tau})=0$. Moreover the
 linearized operator $D_{\bvphi} {\bf
   G}(\tau,\boldsymbol\phi_{\epsilon,\tau}):K^\perp \to K^\perp$ is defined by
  \begin{equation*}
    D_{\bvphi} {\bf G}(\tau,\boldsymbol{\phi}_{\epsilon,\tau})[\bvphi]
    = \bvphi - \Pi^\perp\bigl\{
    D_{\mathbf{u}}\mathbf{F}\(\mathbf{W}
    + \boldsymbol{\phi}_{\epsilon,\tau}\)[\bvphi]
    \bigr\},
  \end{equation*}
  where the linear operator $D_{\mathbf{u}}\mathbf{F}\(\mathbf{u}\):H\to H$ is defined by
  \begin{equation*}
    \bigl(D_{\mathbf{u}}\mathbf{F}(\mathbf{u})[{\bf v}]\bigr)_i
    = {\mathtt i}^*\left\{
      \sum\limits_{j=1}^k\beta_{ij}\(u_j^2v_i+2u_iu_jv_j\)
      - \epsilon^2 V_i(\epsilon{\,\cdot} +\xi)v_i \right\},
    \quad i=1,\dots,k.
  \end{equation*}
  We claim that the operator
  $D_{\bvphi} {\bf G}(\tau,\boldsymbol\phi_{\epsilon,\tau})$ is
  invertible. Indeed, using~\eqref{norm-es-1}, one shows that
  $D_{\mathbf{u}}\mathbf{F}(\mathbf{u}) \to
  D_{\mathbf{u}}\mathbf{F}(\mathbf{U})$ in $\mathcal{L}(H; H)$ as
  $\mathbf{u} \to \mathbf{U}$.  Thus, thanks to \eqref{stima-c0},
  $D_{\bvphi} {\bf G}(\tau,\boldsymbol{\phi}_{\epsilon,\tau}) \to
  D_{\bvphi} {\bf G}(\tau,\mathbf{U}) = \Pi^\perp\boldsymbol{\mathcal{L}}_0
  = \boldsymbol{\mathcal L}_0$ in $\mathcal{L}(K^\perp; K^\perp)$,
  uniformly w.r.t.\ $\tau \in T$, as $\epsilon \to 0$, where
  $\boldsymbol{\mathcal L}_0$ is defined by~\eqref{eq:def-L0}.
  Taking if necessary $\epsilon_0$ smaller, the claim is proved.

  \medskip

(v) Finally, we prove the estimate \eqref{stima-c1}. We know that
$${\bf G}(\tau,\boldsymbol\phi_{\epsilon,\tau})=0.$$
Then, differentiating at $\tau_0$ in the direction $\tau$ yields
$$D_\tau {\bf G}(\tau_0,\boldsymbol\phi_{\epsilon,\tau_0})[\tau]+D_{\mathbf{u}} {\bf G}(\tau_0,\boldsymbol\phi_{\epsilon,\tau_0})\left[D_\tau\boldsymbol\phi_{\epsilon,\tau_0}[\tau]\right]=0$$
and so we get
\begin{equation*}
  \|D_\tau\boldsymbol\phi_{\epsilon,\tau_0}[\tau]\|
  \le C\|D_\tau {\bf G}(\tau_0,\boldsymbol\phi_{\epsilon,\tau_0})[\tau]\|
  \le C\epsilon^3|\tau|,  
\end{equation*}
because (setting $\boldsymbol\phi_{\epsilon,\tau_0}=(\phi_1,\dots,\phi_k)$)
  \begin{equation*}
    \begin{split}
      D_\tau {\bf G}(\tau_0,\boldsymbol{\phi}_{\epsilon,\tau_0})[\tau]
      &= \epsilon^3\Pi^\perp\bigl(
        \mathtt{i}^*\bigl\{
        \nabla V_1(\epsilon{\,\cdot} +\epsilon\tau_0 + \xi_0)
        \tau (W_1 + \phi_1) \bigr\},
        \dots,\\
      &\hspace*{10em}
        \mathtt{i}^*\bigl\{
        \nabla V_k(\epsilon{\,\cdot} +\epsilon\tau_0 +\xi_0)
        \tau (W_k + \phi_k) \bigr\}
        \bigr)
    \end{split}
  \end{equation*}
  and
  \begin{equation*}
    \bigl\|D_\tau {\bf G}(\tau_0,\boldsymbol{\phi}_{\epsilon,\tau_0})[\tau]\bigr\|
    \le C \epsilon^3|\tau|\|\mathbf{W} + \boldsymbol{\phi}_{\epsilon,\tau_0}\|
    \le C \epsilon^3|\tau|.
    \qedhere
  \end{equation*}
\end{proof}

\subsection{Solving (\ref{sc32-1})}

\begin{proposition}\label{punti}
  There exists $\epsilon_0>0$ such that for any
  $\epsilon\in [0,\epsilon_0]$ there exists a unique $\tau_\epsilon\in\R^N$
  such that equation \eqref{sc32-1} is satisfied
  with $\boldsymbol{\phi} = \boldsymbol{\phi}_{\epsilon,\tau}$,
  where $\boldsymbol{\phi}_{\epsilon,\tau}$ is
  given by Proposition~\ref{phi}.
  Moreover, the map $\epsilon \mapsto \tau_\epsilon$ is continuous
  and goes to $0$ as $\epsilon \to 0$.
\end{proposition}

\begin{proof}
  As $\boldsymbol{\phi}_{\epsilon,\tau}$ solves \eqref{sc31-1}, there exist real
  numbers $c^i_{\epsilon,\tau}$, $i=1,\dots,N$ such that (see~\eqref{zi})
  \begin{equation}\label{pu1}
    \boldsymbol{\mathcal L}_{\epsilon,\tau}(\boldsymbol{\phi}_{\epsilon,\tau})
    - \boldsymbol{\mathcal N}_{\epsilon,\tau}(\boldsymbol{\phi}_{\epsilon,\tau})
    - \boldsymbol{\mathcal E}_{\epsilon,\tau}
    = \sum\limits_{i=1}^N c^i_{\epsilon,\tau}\boldsymbol{\Phi}_i.
  \end{equation}

  We aim to find a unique point $\tau=\tau_\epsilon$ such that all the
  $c^i_{\epsilon,\tau}$'s are zero.  We multiply (\ref{pu1}) by
  $\boldsymbol\Phi_j$. We get
  \begin{equation}\label{pu2}
    \left\langle
      \boldsymbol{\mathcal L}_{\epsilon,\tau}(\boldsymbol{\phi}_{\epsilon,\tau})
      - \boldsymbol{\mathcal N}_{\epsilon,\tau}(\boldsymbol{\phi}_{\epsilon,\tau})
      - \boldsymbol{\mathcal E}_{\epsilon,\tau},\boldsymbol{\Phi}_j
    \right\rangle
    = \sum_{i=1}^Nc^{i}_{\epsilon,\tau}\left\langle
      \boldsymbol{\Phi}_i,\boldsymbol{\Phi}_j\right\rangle
    = c^{j}_{\epsilon,\tau}A,
  \end{equation}
  because, by \eqref{zi} and the oddness of $\partial_iU$ along
  the axis $x_i$,
  \begin{equation*}
    \begin{aligned}
      \left\langle\boldsymbol{\Phi}_i,\boldsymbol{\Phi}_j\right\rangle
      &=\sum\limits_{\ell=1}^k  \langle\partial_i U_\ell,\partial_j U_\ell\rangle
        = \sum\limits_{\ell=1}^k  \sigma_\ell^2
        \langle\partial_i U , \partial_j U\rangle
        = \sum\limits_{\ell=1}^k  \sigma_\ell^2
        \int\limits_{\R^N} g'(U)\partial_i U\partial_j U
      \\&   =A \delta_{ij },
      \quad\text{where }
      A:=  \sum\limits_{\ell=1}^k  \sigma_\ell^2
      \int\limits_{\R^N}g'(U)\(\partial_1 U\)^2 \ne 0
      \text{ and } g(t)=t^{3}.
    \end{aligned}
  \end{equation*}
  The claim will follow at once if we prove that
  \begin{equation}\label{chiave}
    \left\langle
      {\boldsymbol{\mathcal L}}_{\epsilon,\tau}(\boldsymbol{\phi}_{\epsilon,\tau})
      - {\boldsymbol{\mathcal N}}_{\epsilon,\tau}(\boldsymbol{\phi}_{\epsilon,\tau})
      - {\boldsymbol{\mathcal E}} _{\epsilon,\tau},\boldsymbol{\Phi_j}
    \right\rangle
    = -\frac{1}{2}
    \epsilon^4 \Biggl( \tau_j \frac{\partial^2\Gamma}{\partial x_j^2} (\xi_0)
    + \order(1)\Biggr),
    \quad j = 1,\dotsc, N,
  \end{equation}
  where the $\order$'s are $C^1$-uniform with respect to $\tau \in T$
  as $\epsilon \to 0$, where
  $T$ is a given compact set.
  Indeed, using~\eqref{chiave}, \eqref{pu2} may be rewritten as
  \begin{equation}\label{add-5}
    -\frac{1}{2}
    \epsilon^4 \Biggl( \tau_j \frac{\partial^2\Gamma}{\partial x_j^2} (\xi_0)
    + \order(1)\Biggr)
    = A c^{j}_{\epsilon,\tau}
    \quad\text{for any}\ j=1,\dots,N,
  \end{equation}
  where $A\not=0$ is a constant  and the $\order(1)$ is $C^1$-uniform
  in $\tau \in T$ as $\epsilon \to 0$.
  Since all the $\partial^2\Gamma/\partial x_j^2 (\xi_0)$'s are different from zero (because
  $\xi_0$ is a non-degenerate critical point of $\Gamma$),
  a contraction mapping argument shows that 
  \begin{equation*}
    \tau_j
    = - \biggl(\frac{\partial^2\Gamma}{\partial x_j^2} (\xi_0) \biggr)^{-1}
    \order(1),
    \quad j = 1,\dotsc, N,
  \end{equation*}
  has a unique solution $\tau_\epsilon = (\tau_1,\dotsc, \tau_N)$
  for $\epsilon$ small enough.
  Therefore the left hand side in \eqref{add-5} vanishes and
  $c^j_{\epsilon,\tau_\epsilon}=0$ for all $j=1,\dots,N$.
  The map $\epsilon \mapsto \tau_\epsilon$ is continuous because
  the functions in $\order(1)$ are continuous with respect to $\epsilon$.
  Moreover it
  is clear that $\tau_j \to 0$, $j = 1,\dotsc, N$, as $\epsilon \to 0$
  and $(\epsilon, \tau) \in \mathcal T$.

  Let us prove \eqref{chiave}. First of all, we estimate the leading
  term in \eqref{pu2}.  By \eqref{zi}, \eqref{eq:scalar-product},
  \eqref{ei}
  (taking into account that $U$ and
  $Z_i$ are radial functions and the derivatives
  $\partial U/\partial x_j$ are odd functions),
  and the Taylor expansions of $V_\ell$ and of
  \begin{equation*}
    \frac{\partial V_\ell}{\partial x_j}(\epsilon x+\epsilon\tau+\xi_0)
    = \frac{\partial V_\ell}{\partial x_j}(\xi_0)
    + \epsilon \sum\limits_{i=1}^N
    \frac{\partial^2 V_\ell}{\partial x_j\partial x_i} (\xi_0) (x_i+  \tau_i)
    + \mathcal O\bigl(\epsilon^2(1+|x|^2)\bigr),
  \end{equation*}
  we get:
  \allowdisplaybreaks
  \begin{align}
      - \langle {\boldsymbol{\mathcal E}}_{\epsilon,\tau},
      \boldsymbol{\Phi}_j \rangle \hspace*{-4em}\nonumber\\
      &=-\sum\limits_{\ell=1}^k\sigma_\ell
        \left  \langle  {\mathcal E}_\ell,
        {\partial U \over \partial x_j} \right\rangle \nonumber\\
      &=\sum\limits_{\ell=1}^k\sigma_\ell\int\limits_{\R^N}
        V_\ell(\epsilon x+\epsilon\tau+\xi_0)
        \bigl(\epsilon^2\sigma_\ell U
        - \epsilon^4 Z_\ell\bigr) \frac{\partial U}{\partial x_j}\intd x
        \nonumber\\
      &= \epsilon^2\sum_{\ell=1}^k\sigma_\ell^2\int\limits_{\R^N}
        V_\ell(\epsilon x+\epsilon\tau+\xi_0)
        U \frac{\partial U}{\partial x_j}\intd x
        + \order(\epsilon^4)  \nonumber\\
      &= - \frac12\epsilon^3\sum\limits_{\ell=1}^k\sigma_\ell^2
        \int\limits_{\R^N} \frac{\partial V_\ell}{\partial x_j}
        (\epsilon x+\epsilon\tau+\xi_0)  U^2(x)\intd x
        + \order(\epsilon^4)  \nonumber\\
      &=-\frac12\epsilon^3\underbrace{
        \sum\limits_{\ell=1}^k\sigma_\ell^2\int\limits_{\R^N}
        \frac{\partial  V_\ell}{\partial x_j} (\xi_0)
        U^2(x)\intd x}_{= \frac{\partial\Gamma}{\partial x_j}  (\xi_0) =0
        \ \text{(see \eqref{gama})}}
        - \frac12\epsilon^4 \sum\limits_{\ell=1}^k\sigma_\ell^2
        \sum_{i=1}^N \frac{\partial^2 V_\ell}{\partial x_j\partial x_i} (\xi_0)
        \tau_i \int\limits_{\R^N}U^2\intd x
        + \order(\epsilon^4)  \nonumber\\
      &= -\frac12\epsilon^4\sum\limits_{i=1}^N\tau_i
        \underbrace{\Biggl(\gamma\sum\limits_{\ell=1}^k\sigma_\ell^2
        \frac{\partial^2 V_\ell}{\partial x_j\partial x_i}  (\xi_0) \Biggr)}_{
        = {\partial^2\Gamma\over\partial x_j\partial x_i} (\xi_0)}
        {} + \order(\epsilon^4)
        = -\frac12\epsilon^4 \tau_j
        \frac{\partial^2\Gamma}{\partial x_j^2} (\xi_0)
        + \order(\epsilon^4).
        \label{pro-error}
  \end{align}
  Note that since $V$ is $C^4$ in a neighbourhood of $\xi_0$, all
  $\order(\epsilon^4)$ hold in the $C^1$-topology
  with respect to  $\tau \in T$.

  Finally, it remains to prove that
  \begin{equation}\label{add-4}
    \langle  {\boldsymbol{\mathcal L}} _{\epsilon,\tau}(\boldsymbol\phi_{\epsilon,\tau}),\boldsymbol\Phi_j\rangle
    = \order\bigl(\epsilon^{4}\bigr)
    \quad \text{and}\quad
   \langle  {\boldsymbol{\mathcal N}} _{\epsilon,\tau}(\boldsymbol\phi_{\epsilon,\tau}),\boldsymbol\Phi_j\rangle
    = \order\bigl(\epsilon^{4}\bigr),
  \end{equation}
  where all $\order(\epsilon^4)$ are in the $C^1(T)$-topology.

  Writing as usual
  $\boldsymbol{\phi}_{\epsilon,\tau} = (\phi_1, \dotsc, \phi_k)$,
  recalling the definition \eqref{elle}, and taking into account that
  $\Phi_j$ solves \eqref{linear-eq} yields
  \begin{align*}
    \langle \boldsymbol{\mathcal{L}}_{\epsilon,\tau}(
    \boldsymbol{\phi}_{\epsilon,\tau}),\boldsymbol{\Phi}_j\rangle
    &= \sum\limits_{\ell=1}^k\sigma_\ell\int\limits_{\mathbb R^N}
      \sum\limits_{\kappa=1}^k \beta_{\ell \kappa}
      \( U_\kappa^2\phi_\ell+2U_\ell U_\kappa\phi_\kappa\) \partial _j
      U\\
    &\qquad-\sum\limits_{\ell=1}^k\sigma_\ell\int\limits_{\mathbb R^N}
      \sum\limits_{\kappa=1}^k \beta_{\ell \kappa}
      \( W_\kappa^2\phi_\ell+2W_\ell W_\kappa\phi_\kappa\)\partial _j
      U\\
    &\qquad+ \epsilon^2\sum\limits_{\ell=1}^k\sigma_\ell
      \int\limits_{\mathbb R^N}V_\ell(\epsilon x+\epsilon\tau+\xi_0)
      \phi_\ell  \, \partial_j U\\
    &=\sum\limits_{\ell=1}^k\sigma_\ell\int\limits_{\mathbb R^N}
      \sum\limits_{\kappa=1}^k \beta_{\ell \kappa}
      \(2\epsilon^2 U_\kappa Z_\kappa -\epsilon^4 Z_\kappa^2\)
      \phi_\ell \, \partial_j U\\
    &\qquad+\sum\limits_{\ell=1}^k\sigma_\ell\int\limits_{\mathbb R^N}
      \sum\limits_{\kappa=1}^k 2 \beta_{\ell \kappa}
      \(\epsilon^2 (U_\ell Z_\kappa + U_\kappa Z_\ell)
      - \epsilon^4 Z_\ell Z_\kappa\)
      \phi_\kappa \, \partial _j U\\
    &\qquad+ \epsilon^2\sum\limits_{\ell=1}^k\sigma_\ell
      \int\limits_{\mathbb R^N}V_\ell(\epsilon x+\epsilon\tau+\xi_0)
      \phi_\ell  \, \partial_j U\\
    &=\order(\epsilon^4)
      \text{ in } C^0(T) \text{ because of \eqref{stima-c0}}.
  \end{align*}
  The derivative with respect to $\tau_j$ of the previous quantity
  enjoys the following estimate:
  \allowdisplaybreaks
  \begin{align*}
    \partial_{\tau_i} \langle \boldsymbol{\mathcal{L}}_{\epsilon,\tau}(
    \boldsymbol{\phi}_{\epsilon,\tau}),\boldsymbol{\Phi}_j\rangle
    &=\sum\limits_{\ell=1}^k\sigma_\ell\int\limits_{\mathbb R^N}
      \sum\limits_{\kappa=1}^k \beta_{\ell \kappa}
      \(2\epsilon^2 U_\kappa Z_\kappa - \epsilon^4 Z_\kappa^2\)
      \partial_{\tau_i}\phi_\ell \, \partial_j U\\
    &\qquad+\sum\limits_{\ell=1}^k \sigma_\ell\int\limits_{\R^N}
      \sum\limits_{\kappa=1}^k 2\beta_{\ell \kappa}
      \(\epsilon^2 (U_\ell Z_\kappa+U_\kappa Z_\ell)
      - \epsilon^4 Z_\ell Z_\kappa\)
      \partial_{\tau_i} \phi_\kappa \, \partial_j U\\
    &\qquad+ \epsilon^2\sum\limits_{\ell=1}^k\sigma_\ell
      \int\limits_{\mathbb R^N}V_\ell(\epsilon x+\epsilon\tau+\xi_0)
      \partial_{\tau_i}\phi_\ell \, \partial _j U\\
    &\qquad+ \epsilon^3\sum\limits_{\ell=1}^k\sigma_\ell
      \int\limits_{\mathbb R^N}\partial_i V_\ell(\epsilon x+\epsilon\tau+\xi_0)
      \phi_\ell \, \partial_j U\\
    &=\order(\epsilon^4)
      \text{ in } C^0(T) \text{ because of \eqref{stima-c1}
      and \eqref{stima-c0}}.
  \end{align*}
  By \eqref{enne},
  \begin{align*}
    \langle  {\boldsymbol{\mathcal N}} _{\epsilon,\tau}(
    \boldsymbol{\phi}_{\epsilon,\tau}),\boldsymbol{\Phi}_j\rangle
    &=\sum\limits_{\ell=1}^k\sigma_\ell\int\limits_{\R^N}
      \sum\limits_{\kappa=1}^k\beta_{\ell \kappa}
      \(W_\ell\phi_\kappa^2+\phi_\ell\phi_\kappa^2+2W_\kappa\phi_\ell\phi_\kappa
      \)\partial _j U\\
    &=\order(\epsilon^4)
      \text{ in } C^0(T) \text{ because of \eqref{stima-c0}}
  \end{align*}
  and, differentiating with respect to $\tau_i$, we easily get
  \begin{equation*}
    \partial_{\tau_i} \langle
    {\boldsymbol{\mathcal N}}_{\epsilon,\tau}(\boldsymbol\phi_{\epsilon,\tau}),
    \boldsymbol\Phi_j\rangle
    = \order(\epsilon^4)
    \text{ in } C^0(T) \text{ because of \eqref{stima-c0} and \eqref{stima-c1}}.
    \qedhere
  \end{equation*}
\end{proof}

\begin{proof}[Proof of Theorem \ref{scexi-1}, completed]
  The existence of the solution $\bf u_\epsilon$ to problem \eqref{sch-1}
  follows combining all the previous arguments.
  It suffices to define
  \begin{equation*}
    \mathbf{u}_\epsilon(x)
    := \bf W\Bigl(\frac{x-\xi_\epsilon}{\epsilon} \Bigr)
    + (\psi_{1,\epsilon}, \dotsc, \psi_{k,\epsilon})(x),
    \quad
    \text{where }
    (\psi_{1,\epsilon}, \dotsc, \psi_{k,\epsilon})(x)
    := \boldsymbol{\phi}_{\epsilon,\tau_\epsilon}
    \Bigl(\frac{x-\xi_\epsilon}{\epsilon} \Bigr)
  \end{equation*}
  with $\xi_\epsilon := \epsilon \tau_\epsilon + \xi_0$, where
  $\tau_\epsilon$ is given by Proposition~\ref{punti}.
  The continuity in the $H^1$-topology results from the fact that,
  for $\epsilon > 0$, the norm $u \mapsto \epsilon^{-N/2}
  \bigl(\int \epsilon^2 |\nabla u|^2 + u^2 \bigr)^{1/2}$
  is equivalent to the usual $H^1$-norm.  Finally the estimate on
  $\psi_{i,\epsilon}$ results from \eqref{stima-c0}.
\end{proof}

\section{The  mass of  $u_\epsilon$}\label{3}

In this section we find the solutions to \eqref {p} by selecting the
solutions to \eqref{sch-1} for a suitable ranges of prescribed masses
$\mu$'s.

\subsection{The non-critical case}

\begin{theorem}
  \begin{itemize}
  \item[(i)] If $N=1$, then there exists $R>0$ such that for any
    $\mu>R$ problem \eqref{p} has a solution $(\epsilon_\mu, \mathbf{u}_\mu),$
    where $\mathbf{u}_\mu$ is concentrating at the point $\xi_0$ as
    $\mu\to\infty$.
  \item[(ii)] If $N=3$, then there exists $r>0$ such that for any
    $\mu<r$ problem \eqref{p} has a solution $(\epsilon_\mu, \mathbf{u}_\mu),$
    where $u_\mu$ is concentrating at the point $\xi_0$ as $\mu\to0$.
  \end{itemize}
  In both cases $\epsilon_\mu^{2 - N} \mu \to \mu_0$ (see \eqref{so}),
  as $\mu \to \infty$ or $\mu \to 0$, respectively.
\end{theorem}

\begin{proof}
  By Theorem \ref{scexi-1} there exists  a solution
  \begin{equation*}
    \mathbf{u}_\eps
    = \mathbf{U} \({x-\xi_\epsilon\over\epsilon}\)
    - \epsilon^2 \mathbf{Z} \({x-\xi_\epsilon\over\epsilon}\)
    + \boldsymbol{\phi}_\epsilon \({x-\xi_\epsilon\over\epsilon}\),
  \end{equation*}
  where $\|\boldsymbol{\phi}_{\epsilon}\| = \mathcal{O}(\epsilon^{3})$.
  The mass of $\mathbf{u}_\epsilon$ is (see \eqref{sps1-mass})
  \begin{align}
    \mu_\epsilon :
    &=
      \epsilon^{-{2}}  \sum\limits_{i=1}^k \, \int\limits_{\R^N}
      \( \sigma_i U \({x-\xi_\epsilon\over\epsilon}\)
      - \epsilon^2 Z_i \({x-\xi_\epsilon\over\epsilon}\)
      + \phi_{i,\epsilon}\(\frac{x-\xi_\epsilon}{\epsilon}\) \)^2
      \intd x  \nonumber\\
    & = \epsilon^{-2+N} \Biggl( \;
     \sum\limits_{i=1}^k \sigma_i^2 \int\limits_{\R^N} U^2(x)  \intd x
      + \mathcal O(\epsilon^2)  \Biggr) \nonumber\\
    &= \epsilon^{-2+N} \bigl(\mu_0+\order(1)\bigr)
      \qquad\text{(see \eqref{so})}.
      \label{cru1sc-1}
  \end{align}
  Since $\mathbf{u}_\epsilon$ must have a prescribed mass as in
  \eqref{sps1-mass}, we
  have to find $\epsilon=\epsilon(\mu)$ such that
  $$\mu_\epsilon = \mu.$$
  As the map $\epsilon \mapsto \mathbf{u}_\epsilon$ is continuous,
  so is $\epsilon \mapsto \mu_\epsilon$.
  Moreover, \eqref{cru1sc-1} implies that
  $\mu_\epsilon \to +\infty$ if $N=1$ (resp.\
  $\mu_\epsilon \to 0$ if $N=3$) as $\epsilon \to 0$.  The Intermediate Value
  Theorem implies that a set of the form $(R, +\infty)$ (resp.\
  $(0, r)$) is in the image of $\epsilon \mapsto \mu_\epsilon$.
\end{proof}

\subsection{The critical case}
Let $N=2$.  It is important to point out that a refinement of the
ansatz is needed! Indeed, if we expand more carefully
\eqref{cru1sc-1}, we can determine the coefficient of the next order
of $\epsilon$:
\begin{equation*}
  \mu_\epsilon =
  \mu_0-2\epsilon^2\underbrace{\sum\limits_{i=1}^k\sigma_i
    \int\limits_{\R^2}U(x) Z_i(x) \intd  x}_{=:\Xi(\xi_0)}
  + \order(\epsilon^2)
  = \mu.
\end{equation*}
This does not allow to determine which values $\mu_\epsilon$ takes
because $\Xi(\xi_0)=0$ as proved in Remark~\ref{xixi}.

\begin{remark}\label{erre0}
  Let $(f_1,\dots,f_k)\in \bigl(L^2(\R^N)\bigr)^k$ and
  $\mathbf{R} = (R_1,\dotsc, R_k)$ a solution to the linear system
  \begin{equation}\label{erre}
    -\Delta R_i+ R_i - \sum_{j=1}^k \beta _{ij}\bigl(U_j^2R_i+2U_iU_jR_j\bigr)
    = f_i .
  \end{equation}
  Consider $z := \sum_{i=1}^k \sigma_i R_i$.  Multiplying the $i$-th
  equation by $\sigma_i$ and summing them up, we find that
  \begin{equation*}
    -\Delta z + z
    - U^2 \sum_{i=1}^k \sigma_i \sum_{j=1}^k \beta_{ij} \sigma_j^2 R_i
    - 2 U^2 \sum_{i=1}^k \sum_{j=1}^k \beta_{ij} \sigma_i^2 \sigma_j R_j
    = \sum_{i=1}^k \sigma_i f_i.
  \end{equation*}
  Making use of \eqref{ccc} and of the fact that
  $\beta_{ij} = \beta_{ij}$, this boils down to
  \begin{equation}\label{erre2}
    -\Delta z+z-3U^2 z
    = \sum_{i=1}^k \sigma_i f_i   \quad\text{in } \R^N.
  \end{equation}
  Therefore, the $L^2$ scalar product of $\mathbf{U}$ with
  $\mathbf{R}$ is
  \begin{align}
    \sum_{i=1}^k \, \int\limits_{\R^N} U_i R_i
    = \int\limits_{\R^N} \sum\limits_{i=1}^k \sigma_iU R_i
    = \int\limits_{\R^N} U(x)z(x)\intd x.
    \label{erre3}
  \end{align}
\end{remark}

\begin{remark} \label{xixi}
  It holds true that $\Xi(\xi_0)=0$.  Indeed we apply the previous remark with
   $f_i= V_i (\xi_0) U_{i}$ (see \eqref{ans-cor}).  Therefore by \eqref{erre3}
  \begin{equation*}
    \Xi(\xi_0)
    = \sum_{i=1}^k \sigma_i \int\limits_{\R^2}U(x) Z_i(x) \intd x
    = \int\limits_{\R^2}U(x)z(x)\intd x,
  \end{equation*}
  where $z$ solves
  \begin{equation*}
    -\Delta z+z-3U^2 z
    =\Biggl(\sum\limits_{\ell=1}^k\sigma_\ell^{2} V_\ell (\xi_0)\Biggr) U
    \quad\text{in}\ \R^2.
  \end{equation*}
  A direct computation shows that
  \begin{equation}
    \label{eq:fml-z}
    z(x) = -\frac12\Biggl(\sum_{\ell=1}^k\sigma_\ell^2 V_\ell (\xi_0)\Biggr)
    \bigl(U(x) + \nabla U(x) \cdot x\bigr).
  \end{equation}
  Indeed, letting $u_{\lambda}(x) := \lambda U(\lambda x)$, it is easy
  to check that $u_{\lambda}(x)$ satisfies
  \begin{equation}\label{1-1}
    -\Delta u_{\lambda}+\lambda^{2}u_{\lambda}=u_{\lambda}^{3}
    \quad\text{in } \R^{2}.
  \end{equation}
  Differentiating \eqref{1-1} with respect to $\lambda$ and
  taking $\lambda=1$,
  we deduce that the function
  $v=\frac{\partial u_{\lambda}}{\partial \lambda}|_{\lambda=1}
  =U + \nabla U \cdot x$ satisfies
  \begin{equation*}
    -\Delta v+v -3U^{2}v=-2U
    \quad \text{in } \R^2.
  \end{equation*}
  Equality~\eqref{eq:fml-z} then results for the fact that
  the kernel of $v \mapsto -\Delta v+v -3U^{2}v$ is spanned by
  $\partial_iU$, $i = 1,2$, and so the operator
  is injective on radial functions.
  Next, integrating by parts (for example, using
  $\Div(\frac{1}{2} U^2 x) = U^2 + U \nabla U \cdot x$),
  it is immediate to check that
  \begin{equation*}
    \int_{\R^{2}} \bigl(U + \nabla U \cdot x\bigr) U(x) \intd x=0.
  \end{equation*}
\end{remark}

\medskip

We agree that a refinement of the ansatz is necessary.
Without loss of generality,  we may assume the common critical point of the potentials is $\xi_0=0$ and also that each potential vanishes at $\xi_0,$ 
 i.e.\ in a neighbourhood of the origin
\begin{equation}\label{viii}
  V_i(x) = \bigl(a^{(i)}_1x_1 ^2+a^{(i)}_2x_2^2\bigr)
  + \mathcal O(|x|^3)\quad \text{with}\ a^{(i)}_1,a^{(i)}_2\in\R,\ \hbox{for every}\ i=1,\dots,k.
\end{equation}
Set
\begin{equation}\label{alfa}
  \alpha:=\int\limits_{\R^N}U(x)z_0(x)\intd x,
\end{equation}
where $z_0$ is the radial solution to
\begin{equation}\label{z0}
  -\Delta z_0 +z_0-3U^2 z_0 = \bigl(x_1^2+x_2^2\bigr) U \text{ in } \R^2.
\end{equation}
In Section \ref{numeric}, there is a numerical evidence that $\alpha>0.$
\begin{theorem}\label{th:main3crit}
 There
  exists $\delta>0$ such that if either
  $  \Delta \Gamma(0) >0$ and
  $\mu\in\(\mu_0-\delta,\mu_0\)$, or $ \Delta \Gamma(0) <0$ and
  $\mu\in\(\mu_0,\mu_0+\delta\)$, problem \eqref{p} has a solution
  $(\epsilon_\mu, u_\mu)$ such that $u_\mu$ concentrates at
  the origin and
    $\epsilon_\mu^{-4} (\mu_0 - \mu) \to \alpha \, \Delta  \Gamma(0) $
  as $\mu\to\mu_0.$
\end{theorem}

\begin{proof}
   We briefly sketch the main steps of the proof, which relies on the
   same arguments used in the non-critical case.

   We look for a solution to \eqref{p}, where we choose
   $\xi=\epsilon^2\tau$ as
   \begin{equation}
     \label{ans2}
     \mathbf{u}
     = \underbrace{\mathbf{U} -\epsilon^4 \mathbf{Q}}_{=:\mathbf{W}}
     + \boldsymbol\phi,
   \end{equation}
   where $\mathbf{U} = (U_1,\dots,U_k)$ is the non-degenerate
   synchronized solution of the limit system \eqref{limsps},  the
   remainder term $\boldsymbol{\phi} \in K^\perp$ (see \eqref{Kperp}) and
   the second order correction term $\mathbf{Q}  := (Q_1,\dots,Q_k)$ solves
   the linear system
   \begin{multline}
     -\Delta Q_i + Q_i
     - \sum_{j=1}^k \beta _{ij}\bigl(U_j^2 Q_i + 2U_iU_j Q_j\bigr)
     \\
     = \(a^{(i)}_1x_1 ^2+a^{(i)}_2x_2^2\) U_i
     \ \text{in}\ \R^2,\quad i=1,\dots,k.
     \label{ans-cor2}
   \end{multline}
   The proof proceeds as in the previous case. Here, it is only a
   matter of noting a couple of crucial facts. First of all, the size
   of the error $\boldsymbol{\mathcal E}_{\epsilon,\tau}$ is
   \begin{equation*}
     \|{\boldsymbol{ \mathcal E}}  _{\epsilon,\tau}\|\lesssim \epsilon^5,
   \end{equation*}
   because (see also \eqref{errore})
   \begin{multline}
     \label{errore2}
    \mathcal E_i
    = \mathtt{i}^*\Biggl\{\epsilon^8\sum\limits_{j=1}^k\beta_{ij}
    \Bigl(2U_j Q_i Q_j + U_i Q_j^2 - \epsilon^4 Q_i Q_j^2\Bigr) \Biggr\}
    \\
    +{\mathtt i}^*\Bigl\{ \epsilon^6V_i(\epsilon x+\epsilon^2\tau) Q_i
    -\epsilon^2\left[V_i(\epsilon x+\epsilon^2\tau) 
        -\epsilon^2\bigl(a^{(i)}_1x_1^2+a^{(i)}_2x_2^2\bigr)\right]
    U_i\Bigr\}.
  \end{multline}
  Next, the component of the error along the element of the kernel is
  (see also \eqref{pro-error})
  \allowdisplaybreaks
  \begin{align*}
    -\langle {\boldsymbol{ \mathcal E}}_{\epsilon,\tau},
    \boldsymbol\Phi_j \rangle
    &=-\sum\limits_{\ell=1}^k\sigma_\ell
      \left\langle {\mathcal E}_\ell,
      \frac{\partial U}{\partial x_j} \right\rangle \\
    &=\epsilon^2\sum\limits_{\ell=1}^k\sigma_\ell^2\int\limits_{\R^2}
      V_\ell(\epsilon x+\epsilon^2\tau )  U
      \frac{\partial U}{\partial x_j}\intd x
      + \order(\epsilon^5)\\
    &=-\frac12\epsilon^3\sum_{\ell=1}^k\sigma_\ell^2\int\limits_{\R^2}
      \frac{\partial  V_\ell}{\partial x_j} (\epsilon x+\epsilon^2\tau )
      U^2(x)\intd x
      + \order(\epsilon^5)\\
    &= - \epsilon^5 \Biggl( \sum\limits_{\ell=1}^k\sigma_\ell^2a^{(\ell)}_j
      \tau_j\int\limits_{\R^2}U^2\intd x
      + \frac{1}{4} \sum_{\ell=1}^k \sum_{i=1}^N \sigma_\ell^2
      \frac{\partial^3 V_\ell}{\partial x_i^2 \partial x_j}(0)
      \int\limits_{\R^2} x_i^2 U^2 \intd x
      \Biggr)
      + \order(\epsilon^5)\\
    &= - \epsilon^5 \Biggl( \tau_j
      \underbrace{\gamma
      \sum\limits_{\ell=1}^k\sigma_\ell^2a^{(\ell)}_j}_{
      = \frac{1}{2} \frac{\partial^2\Gamma}{\partial x_j^2}(0)}
      + \frac{1}{4\gamma} \sum_{i=1}^N
      \frac{\partial^3 \Gamma}{\partial x_i^2 \partial x_j}(0)
      \int\limits_{\R^2} x_i^2 U^2 \intd x
      \Biggr)
      + \order(\epsilon^5)\\
    &= -\frac{1}{2} \epsilon^5 \Biggl( \tau_j
      \frac{\partial^2\Gamma}{\partial x_j^2}(0)
      + \frac{\widetilde\gamma}{2N\gamma}
      \frac{\partial \Delta \Gamma}{\partial x_j}(0)
      + \order(1)
      \Biggr),
  \end{align*}
  where we have used that $\int_{\R^2} x_i^2 U^2 \intd x$ is
  independent of $i$ and so
  \begin{equation*}
    \int\limits_{\R^2} x_i^2 U^2 \intd x
    = \frac{1}{N} \sum_{i=1}^N \int\limits_{\R^2} x_i^2 U^2 \intd x
    = \frac{\widetilde \gamma}{N}
    \qquad\text{with }
    \widetilde\gamma
    := \int\limits_{\R^2} |x|^2 U^2(x) \intd x.
  \end{equation*}
  Since ${\partial^2\Gamma\over\partial x_j^2} (0)\not=0$ for any
  $j=1,2,$ we can argue exactly as in the previous part
  to get the existence of a solutions concentrating at $\xi_0 = 0$ as
  \begin{equation*}
   \mathbf{u}_\eps
    =\mathbf{u} \({x-\xi_\epsilon\over\epsilon}\)
    - \epsilon^4 \mathbf{Q} \({x-\xi_\epsilon\over\epsilon}\)
    + \boldsymbol{\phi}_\epsilon\({x-\xi_\epsilon\over\epsilon}\),
  \end{equation*}
  whose mass is
  \begin{align}
    \mu_\epsilon
    &=
      \epsilon^{-{2}}  \sum\limits_{i=1}^k \int\limits_{\R^2}
      \( \sigma_i U \({x-\xi_\epsilon\over\epsilon}\)
      - \epsilon^4 Q_i \({x-\xi_\epsilon\over\epsilon}\)
      + \phi_{i,\epsilon}\({x-\xi_\epsilon\over\epsilon}\) \)^2
      \intd x  \nonumber\\
    & =
      \mu_0 -2 \epsilon^4
      \underbrace{\int\limits_{\R^2} \,\sum_{i=1}^k
      \sigma_i U(x) Q_i(x)\intd x}_{=:\Upsilon }
      {} +\order(\epsilon^4)\quad \text{(see \eqref{so})}.
  \end{align}
   Since $u_\epsilon$ must have a prescribed mass equal to $\mu$ we
  have to find $\epsilon=\epsilon(\mu)$ such that
  \begin{equation}
    \label{cru2sc}
    \mu_\epsilon =  \mu.
  \end{equation}
  As the map $\epsilon \mapsto \mathbf{u}_\epsilon$ is continuous,
  so is $\epsilon \mapsto \mu_\epsilon$.  Moreover, it
  goes to $\mu_0$ as $\epsilon \to 0$.
  Thanks to Remark \ref{upsilon} (below), its image must contain
  an interval of the form
  $(\mu_0-\delta, \mu_0)$ if $ \alpha \Delta \Gamma (0) > 0$ or
  $(\mu_0, \mu_0+\delta)$ if $ \alpha \Delta \Gamma (0) < 0$.  That
  concludes the proof.
\end{proof}

\begin{remark}\label{upsilon}
  It holds true that $\Upsilon =\frac12 \alpha \Delta \Gamma (0)$
  (see \eqref{alfa}).  We apply Remark \ref{erre0} with
  $f_\ell = \sigma_\ell \bigl(a^{(\ell)}_1x_1 ^2+a^{(\ell)}_2x_2^2\bigr) U$
  (see \eqref{ans-cor2}).  Therefore by \eqref{erre3}
  \begin{equation*}
    \Upsilon = \sum_{i=1}^k \int\limits_{\R^N}U_i(x) Q_i(x) \intd x
    = \int\limits_{\R^N}U(x)z^*(x)\intd x,
  \end{equation*}
  where $z^{*}$ solves
  \begin{equation*}
    -\Delta z+z-3U^2 z
    = \sum_{\ell=1}^k \sigma^2_\ell \bigl(a^{(\ell)}_1x_1^2+a^{(\ell)}_2x_2^2\bigr) U
    \quad \text{in}\ \R^2.
  \end{equation*}
  Now, we can write
  \begin{equation*}
    z^{*} = \(\sum\limits_{\ell=1}^k \sigma^2_\ell   a^{(\ell)}_1\) z_1^*
    + \(\sum\limits_{\ell=1}^k \sigma^2_\ell  a^{(\ell)}_2\)z_2^*,
  \end{equation*}
  where $z_i^*$ solves
  \begin{equation*}
    -\Delta z_i^* +z_i^*-3U^2 z_i^* = x_i^2 U\quad \text{in}\ \R^2.
  \end{equation*}
  We observe that
 $z_2^*(x_1,x_2)=z_1^*(x_2,x_1)$ 
  and so
  \begin{equation*}
    \int\limits_{\R^N}U(x)z_1^*(x)\intd x
    = \int\limits_{\R^N}U(x)z_2^*(x)\intd x
    = \frac12\int\limits_{\R^N}U(x)z_0(x)\intd x,
  \end{equation*}
  where $z_0$ solves \eqref{z0}.  Then
  \begin{align*}
    \int\limits_{\R^N}U(x)z^*(x)\intd x
    &=\(\sum\limits_{\ell=1}^k \gamma \sigma^2_\ell   a^{(\ell)}_1\)
      \int\limits_{\R^N}U(x)z_1^*(x)\intd x
      + \(\sum\limits_{\ell=1}^k \gamma \sigma^2_\ell  a^{(\ell)}_2\)
      \int\limits_{\R^N}U(x)z_2^*(x)\intd x	\\
    &= \frac{\alpha}{2\gamma}
      \underbrace{\gamma \sum_{\ell=1}^k \sigma^2_\ell
      \(a^{(\ell)}_1+a^{(\ell)}_2\)}_{=\frac{1}{2} \Delta \Gamma(0)}
  \end{align*}
  and the claim follows.
\end{remark}

\section{Numerical evidence for the assumption
  $\alpha \ne 0$}
\label{numeric}

Set
\begin{equation}\label{alfan}
\alpha_N:=\int\limits_{\R^N} U(x)S(x)\intd x,
\end{equation}
where $U$ is the positive radial solution of
$$-\Delta U+U=U^p\quad \text{in}\ \R^N$$
and $S$ is the radial solution of
$$-\Delta S+S-pU^{p-1} S=|x|^2 U\quad \text{in}\ \R^N.$$
In the previous section $\alpha=\alpha_2$.
In this section, we would like to provide numerical evidence that
$\alpha_N > 0$ for the $L^2$-critical exponent, i.e.
\begin{equation}
  \label{eq:integUW>0}
  p = 1 + \frac{4}{N}
  \qquad\Rightarrow\qquad
  \int_0^\infty U(r)  S(r) r^{N-1}  \intd  r > 0.
\end{equation}
This was proved in \cite[Remark~3.5]{ppvv} for $N = 1$.  Here we
numerically estimate the integral in \eqref{eq:integUW>0} for larger
values of $N$.  To do this, $U(0)$ is estimated using a bisection
procedure to find the right initial condition that lies between the
set of $U(0)$ such that $\forall r > 0,\linebreak[2]\ U(r) > 0$ and
those such that $U$ has at least one root.  To determine $S(0)$, we
impose that $S(r_0) = 0$ for a ``large'' $r_0$.  Such $S(0)$ is easy
to compute since $S(r_0)$ is an affine function of $S(0)$.  On
Fig.~\ref{fig:integUW_1-4}, you can see the result of these
computations as the graphs of the functions
$p \mapsto \int_0^\infty U S r^{N-1}  \intd  r$ for
$N \in \{1,\dotsc, 8\}$.  The large dot on each curve indicates the
point of the graph for $p = 1 + 4/N$.  These provide clear evidence
that \eqref{eq:integUW>0} holds.

From these graphs, we also conjecture that
\begin{equation*}
  \alpha_N \to 0
  \quad\text{as } p \to 2^* - 1 = \frac{N+2}{N-2}.
\end{equation*}

\begin{figure}[ht]
  \centering
  \begin{tikzpicture}[x=15mm, y=5mm]
    \newcommand{\xmin}{1.3}
    \draw[->] (\xmin, 0) -- (5.3, 0) node[below]{$p$};
    \draw[dotted, thick] (\xmin, 6.53) -- (\xmin, 7.4);
    \draw (\xmin, -18mm) -- (\xmin, 6.5);
    \foreach \x in {2, 3, 4, 5}{
      \draw (\x, 0) -- +(0, 3pt) -- +(0, -3pt) node[below]{$\scriptstyle\x$};
    }
    \foreach \y in {-3, -2, -1, 0, 1, 2, 3, 4, 5, 6}{
      \draw (\xmin, \y) -- +(3pt, 0) -- +(-3pt, 0)
      node[left]{$\scriptstyle\y$};
    }
    \begin{scope}
      \foreach \n in {1, 2, 3, 4}{
        \draw[color=n\n, line width=1pt] (4, 8\baselineskip)
        ++(0, \n\baselineskip) -- +(1em, 0) node[right]{$N = \n$};
      }
      \clip (\xmin, -18mm) rectangle (5.1, 6.5);
      \foreach \n in {1, 2, 3, 4}{
        \draw[color=n\n, line width=1pt] plot file{\gr{integ_UW_N\n.dat}};
      }
      \foreach \n/\p/\i in {1/5/0.41922325133441757,
        2/3/1.1047799217791177, 3/2.33333/4.403048042874518}{
        \fill[color=n\n] (\p, \i) circle(2pt);
      }
    \end{scope}
    \begin{scope}[color=n4, dotted, thick]
      \draw (1.68, 6.4) -- (1.7, 7.5);
      \draw (2.3, 6.45) -- (2.04, 7.5);
    \end{scope}
    \begin{scope}[yshift=18mm, y=1mm]
      \draw[->] (\xmin, 19) -- (\xmin, 40);
      \foreach \y in {20, 25, 30, 35}{
        \draw (\xmin, \y) -- +(3pt, 0) -- +(-3pt, 0)
        node[left]{$\scriptstyle\y$};
      }
      \fill[color=n4] (2, 23.516225342983404) circle(2pt);
      \clip (\xmin, 19) rectangle (5, 38);
      \draw[color=n4, line width=1pt] plot file{\gr{integ_UW_N4.dat}};
    \end{scope}
    \begin{scope}[xshift=0.2\linewidth, x=55mm, y=4mm]
      \draw[->] (\xmin, 0) -- (2.35, 0) node[below]{$p$};
      \foreach \x in {1.4, 1.6, 1.8, 2, 2.2}{
        \draw (\x, 0) -- +(0, 3pt) -- +(0, -3pt) node[below]{$\scriptstyle\x$};
      }
      \draw[dotted, thick] (\xmin, 3.6) -- +(0, 2.6ex);
      \draw (\xmin, -18mm) -- (\xmin, 3.5);
      \foreach \y in {-4, -3, -2, -1, 1, 2, 3}{
        \draw (\xmin, \y) +(3pt, 0) -- +(-3pt, 0)
        node[left]{$\scriptstyle \y 000$};
      }
      \foreach \n in {5, 6, 7, 8}{
        \draw[color=n\n, line width=1pt] (2, 4\baselineskip)
        ++(0, \n\baselineskip) -- +(1em, 0) node[right]{$N = \n$};
      }

      \foreach \n/\p/\i in {5/1.8/0.1575407663466334,
        6/1.6666/1.269877871743919}{
        \fill[color=n\n] (\p, \i) circle(2pt);
      }
      \begin{scope}
        \clip (\xmin, -18mm) rectangle (2.35, 3.6);
        \foreach \n in {5, 6, 7, 8}{
          \draw[color=n\n, line width=1pt] plot file{\gr{integ_UW_N\n.dat}};
        }
      \end{scope}
      \begin{scope}[dotted, thick]
        \draw[color=n7] (1.447, 3.75) -- (1.45, 4.7);
        \draw[color=n7] (1.64, 3.7) -- (1.6, 4.7);
        \draw[color=n8] (1.4, 3.75) -- (1.4, 4.7);
        \draw[color=n8] (1.613, 3.75) -- (1.6, 4.7);
      \end{scope}
      \begin{scope}[yshift=11.2mm, y=1mm]
        \draw (\xmin, 8) -- (\xmin, 27);
        \draw[dotted, thick] (\xmin, 27) +(0, 0.5mm) -- +(0, 2.3ex);
        \foreach \y in {10, 20, 25}{
          \draw (\xmin, \y) +(3pt, 0) -- +(-3pt, 0)
          node[left]{$\scriptstyle \y 000$};
        }
        \fill[color=n7] (1.571, 11.964556975192801) circle(2pt);
        \begin{scope}[dotted, thick, color=n8]
          \draw (1.402, 26.3) -- (1.405, 31);
          \draw (1.56, 26.5) -- (1.53, 31);
        \end{scope}
        \clip (\xmin, 7.5) rectangle (2.35, 26);
        \foreach \n in {7, 8}{
          \draw[color=n\n, line width=1pt] plot file{\gr{integ_UW_N\n.dat}};
        }
      \end{scope}
      \begin{scope}[yshift=38mm, y=0.06mm]
        \draw[->] (\xmin, 70) -- (\xmin, 360);
        \foreach \y in {100, 200, 300}{
          \draw (\xmin, \y) +(3pt, 0) -- +(-3pt, 0)
          node[left]{$\scriptstyle \y 000$};
        }
        \fill[color=n8] (1.5, 129.02646069396114) circle(2pt);
        \clip (\xmin, 70) rectangle (2.35, 300);
        \draw[color=n8, line width=1pt] plot file{\gr{integ_UW_N8.dat}};
      \end{scope}
    \end{scope}
  \end{tikzpicture}
  \caption{Graphs of $p \mapsto \int_0^\infty U S r^{N-1} \intd r$
    for $N \in \{1,\dotsc, 8\}$.}
  \label{fig:integUW_1-4}
\end{figure}
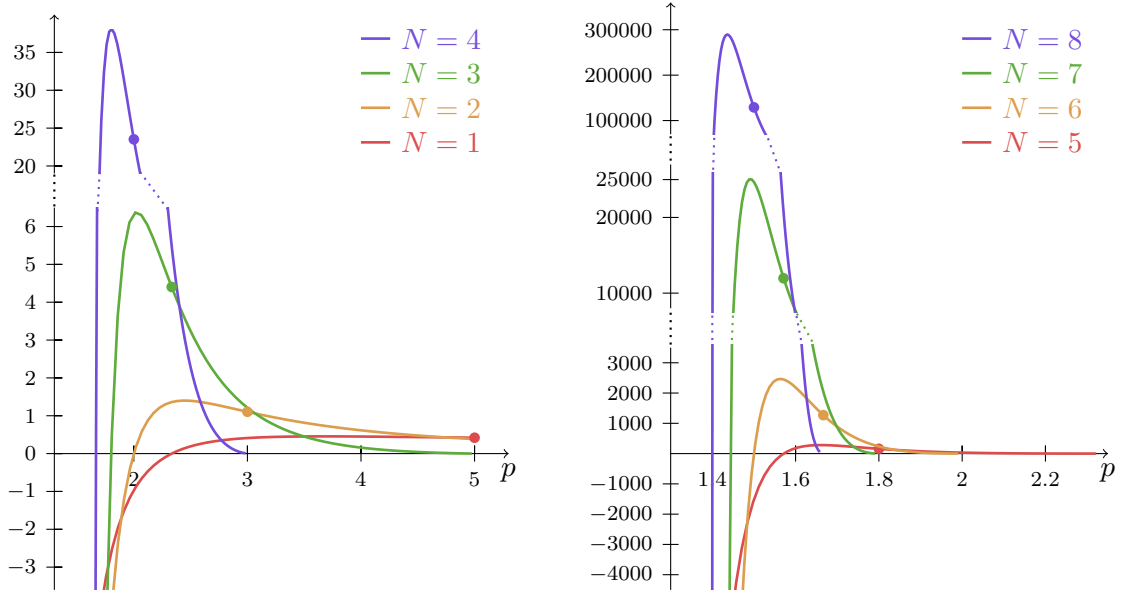

\appendix
\section{A non-degenerate result}\label{app}

In the following we use some ideas introduced in \cite{piso}.

\medskip

Let us consider the eigenvalue problem
\[
-\Delta \psi + \psi= \lambda U^2 \psi\quad \text{in}\ \R^N.
\]
The classical Fredholm alternative Theorem allows to claim that there exists a sequence of positive eigenvalues $\{\lambda_m\}_{m\in\mathbb N}$ with
\[
1=\lambda_1< 3=\lambda_2 < \lambda_3 < \dots < \lambda_m<
\lambda_{m+1} < \cdots \quad \text{and}\quad \lambda_m \to +\infty.
\]
The eigenspace associated to the first eigenvalue is a $1-$dimensional
space generated by the positive function $U.$ Moreover, it is well
known that eigenspace associated to the second eigenvalue is a
$N-$dimensional space generated by
$ \psi_i:={\partial U\over\partial x_i}$ for $  i=1,\dots,N$.

\medskip

We observe that system \eqref{linear-eq} can be rewritten as
\begin{equation}\label{sist3}
  -\Delta \phi_i + \phi_i - (2\beta_{ii}\sigma_i^2+1) U^2\phi_i
  - 2 \sigma_i \sum_{\substack{j=1\\ j\ne i}}^k \beta _{ij}\sigma_j U^2\phi_j
  = 0  \ \text{in}\ \R^N,\quad i=1,\dots,k.
\end{equation}
\begin{equation*}
  -\Delta \mathbf{v} + \mathbf{v}=U^2 {\mathcal M} \mathbf{v}
  \ \text{in}\ \R^N,
\end{equation*}
with
\begin{equation}\label{cccm}
  {\mathcal M}:=\ {\mathcal Id}+2\ {\mathcal C}
  \quad \text{and}\quad
  \mathcal{C}
  := \begin{pmatrix}
    \beta_{11}\sigma_1^2 & \beta_{12}\sigma_1\sigma_2  &\dots& \beta_{1k}\sigma_1\sigma_k\\
    \beta_{12}\sigma_1\sigma_2&   \beta_{22}\sigma_2^2 &  \dots&\beta_{2k}\sigma_2\sigma_k\\
    \vdots                   & \vdots    &  \ddots&\vdots\\
    \beta_{1k}\sigma_1\sigma_k&  \beta_{2k}\sigma_2\sigma_k &  \dots&  \beta_{kk}\sigma_k^2
  \end{pmatrix}.
\end{equation}
Let $\Lambda$ be an eigenvalue of $\mathcal M$ and $e$ an associated
eigenfunction, i.e.
\[
\mathcal M e=\Lambda e.
\]
It is useful to point out that $\Lambda_\ell$ is an eigenvalue of
$\mathcal M$ if and only if $\Theta_\ell:= (\Lambda_\ell-1)/ 2$ is an
eigenvalue of the matrix $\mathcal C$.  It is immediate to check that
$\Theta=1$ is an eigenvalue of $\mathcal C$ whose eigenvector is
$(\sigma_1,\dots,\sigma_k).$ We set $\Theta_1=1$, which implies
$\Lambda_1=3.$

\begin{proposition}\label{prop: base thm 2 general}
  Assume that, all the eigenvalues $\Lambda_2,\dots,\Lambda_k$ of
  $\mathcal{M}$ do not coincide with any of the eigenvalues
  $\{\lambda_m: m \in \mathbb N\}$, i.e.
  \begin{equation}\label{eigenvalues}
    \Lambda_\ell \not\in\{\lambda_1, \dots,\lambda_m,\dots\}
    \ \text{for any}\ \ell=2,\dots,k.
  \end{equation}
  Then the set of solutions to the linear system \eqref{sist3} is
  $N-$dimensional, and is generated by $\psi_i\mathfrak e_1,$ where
  $\mathfrak e_1=(\sigma_1,\dots,\sigma_k)\in \R^k$ is an eigenvector
  associated with $\Lambda_1=3$ and
  $ \psi_i:={\partial U\over\partial x_i}$ for $ i=1,\dots,N.$
\end{proposition}

\begin{proof}
  Let $\Lambda_\ell$ be an eigenvalue of the matrix $\mathcal M$ and
  let $\mathfrak e_\ell\in\R^k$ be an associated eigenvector.  We
  multiply \eqref{sist3} by $\mathfrak e_\ell$ and taking into account
  the symmetry of the matrix $\mathcal M$ we get
  \[
    -\Delta (\mathfrak e_\ell\cdot\mathbf{v})
    + \mathfrak e_\ell\cdot\mathbf{v}
    =\Lambda_\ell\ U^2 (\mathfrak e_\ell\cdot \mathbf{v})\ \text{in}\ \R^N.
  \]
  Since $\Lambda_\ell \neq \lambda_m$ for every $m$, we deduce that
  \begin{equation*}
    \mathfrak  e_\ell\cdot\mathbf{v}=0
    \quad \text{for any}\ \ell=2,\dots,k,
  \end{equation*}
  which implies (by the orthogonality of eigenvectors associated to
  different eigenvalues) that
  $$\mathbf{v}=\psi(x) \mathfrak e_1
  \ \text{for some function $\psi$ such that}\
  -\Delta \psi + \psi=3 U^2 \psi\ \text{in}\ \R^N$$
  and the claim follows.
\end{proof}

A first consequence is the following example.

\begin{example}\label{ex1}
  Let  $k=2.$ The system \eqref{limsps} has a non-degenerate synchronized solution if
  \begin{equation}\label{esempio1}
  \hbox{either}\ -\sqrt{\mu_1 \mu_2} <\beta < \min\{\mu_1,\mu_2\}\ \hbox{or}\
  \beta> \max\{\mu_1,\mu_2\},
  \end{equation}
  where $\mu_i:=\beta_{ii}$ and $\beta:=\beta_{12}=\beta_{21}.$\end{example}

\begin{proof} First of all, we observe that in this case system \eqref{ccc} reduces to
  \begin{equation}\label{systci}
    \begin{cases}
      \mu_1 \sigma_1^2 + \beta \sigma_2^2 = 1, \\
      \beta \sigma_1^2 + \mu_2 \sigma_2^2 = 1,
    \end{cases}
  \end{equation}
 which
  admits the solution
  \begin{equation}\label{def c_i}
    \sigma_1^2:= \frac{\beta-\mu_2}{\beta^2 - \mu_1 \mu_2},
    \quad
    \sigma_2^2:= \frac{\beta-\mu_1}{\beta^2 - \mu_1 \mu_2}
  \end{equation}
  if \eqref{esempio1} holds true.  \\
  Next, to prove that it is non-degenerate, we apply Proposition \ref{prop: base thm 2 general} showing that
  assumption \eqref{eigenvalues} holds.
  The matrix $M:= (\alpha_{ij})_{i,j=1,2}$ in \eqref{cccm} reduces to
  \begin{equation}\label{def alpha}
    \alpha_{11} :=3\mu_1 \sigma_1^2+\beta \sigma_2^2,\quad
    \alpha_{12} = \alpha_{21} := 2\beta \sigma_1 \sigma_2, \quad
    \alpha_{22}:=3\mu_2\sigma_2^2+\beta \sigma_1^2.
  \end{equation}
  Its  eigenvalues are
  \begin{equation}\label{def lambda i}
    \Lambda_1 = \frac{\alpha_{11}+\alpha_{22}
      + \sqrt{(\alpha_{11}-\alpha_{22})^2+4\alpha_{12}^2}}{2},\quad
    \Lambda_2 = \frac{\alpha_{11}+\alpha_{22}
      - \sqrt{(\alpha_{11}-\alpha_{22})^2 + 4\alpha_{12}^2}}{2}.
  \end{equation}
  Using the definition of $\alpha_{ij}$ and that of $\sigma_1,\sigma_2$,
  it is not difficult to check that
  \begin{equation*}
    \Lambda_{2}
    = \frac{6-2\beta (\sigma_1^2+\sigma_2^2)
      - 2\beta (\sigma_1^2+\sigma_2^2)}2
    = 3-2\beta (\sigma_1^2+\sigma_2^2)
  \end{equation*}
  and
  \begin{equation*}
    \Lambda_{1}
    = \frac{6-2\beta (\sigma_1^2+\sigma_2^2) + 2\beta (\sigma_1^2+\sigma_2^2)}2
    =3.
  \end{equation*}
  In particular, $\Lambda_2<\Lambda_1 = 3$ for $\beta>0$. Moreover, a direct
  computation shows that if either $\beta <\min\{\mu_1,\mu_2\}$, or
  $\beta >\max\{\mu_1,\mu_2\}$, then $\Lambda_1 \neq 1.$ Finally, if
  $\Lambda_1=1,$ then it is immediate to check that
  \begin{equation*}
    1 = 3-2\beta \frac{2\beta-\mu_1-\mu_2}{\beta^2 -\mu_1 \mu_2},
  \end{equation*}
  which implies either $\beta = \mu_1$, or $\beta=\mu_2$ which is not
  possible.
\end{proof}

\begin{proposition}\label{prop: on hp 2}
  Suppose that the matrix
  $\mathcal{B} := (\beta_{ij})_{1\le i,j\le k}$ is invertible and has
  only positive elements.  Then the linearized systems \eqref{sist3}
  has a $N$-dimensional set of solutions.
\end{proposition}

\begin{proof}
  As observed above, we have to prove that if $(\beta_{ij})$ is
  invertible and has positive entries, then the set of solutions to
  \eqref{sist3} is $N$-dimensional. By Proposition \ref{prop: base thm
    2 general}, this amounts to show that if $(\beta_{ij})$ is
  invertible and has positive entries, then the eigenvalues
  $\Lambda_2,\dots,\Lambda_k$ of $\mathcal{M}$ are different from
  $\lambda_1=1$, $\lambda_2 = 3$, $\lambda_m >3$.

  Let us argue in terms of the matrix $\mathcal{C}$. By assumption,
  $\mathcal C$ has positive entries. Therefore by Perron-Frobenius
  Theorem we deduce that the eigenvalue $\Theta_1=1$, which is
  associated to the eigenvector of positive elements
  $(\sigma_1,\dots,\sigma_k)$, is simple, and any other eigenvalue
  $\Theta_\ell$ satisfies $|\Theta_\ell|<1$.  Moreover, $0$ is not an
  eigenvalue of the matrix $\mathcal C$, since a straightforward
  computation shows that
  \begin{equation*}
    \det \mathcal C
    = -(\sigma_1^2\cdot\dots\cdot \sigma_k^2) \det(\beta_{ij})
    \ne 0
  \end{equation*}
  being $(\beta_{ij})$ invertible.  Therefore, $\Lambda_1=3$ is a
  simple eigenvalue, and we have that both $-1<\Lambda_\ell<3$ and
  $\Lambda_\ell\not =1$ for any $\ell=2,\dots,k$. This completes the
  proof.
\end{proof}

\begin{example} \label{ex2}
  Let  $k\ge2.$ The system \eqref{limsps} has a non-degenerate synchronized solution if
  \begin{equation}\label{esempio2}
 0<\mu_1<\dots<\mu_k,\ \beta:=\beta_{ij}\  \hbox{for every}\ i\not=j \ \hbox{and}\ \beta\ \hbox{is large enough,}
  \end{equation}
  where  $\mu_i:= \beta_{ii}.$
\end{example}

\begin{proof}  The  existence  of a synchronized solution is
proved in \cite{ba-2013}, once we choose
  $0<\mu_1<\dots<\mu_k$ and $\beta_{ij}=\beta$ for every  $i\not=j $ with
  $\beta>\mu_k.$ Indeed it is easy to check that
  \begin{equation*}
    \sigma_i:=\left[(\mu_i-\beta)
        \Biggl(1+\beta\sum_{j=1}^k \frac{1}{\mu_j -\beta}\Biggr)
      \right]^{-1/2}
    \quad \text{for}\  i=1,\dots,k
  \end{equation*}
  is a solution to system \eqref{ccc}. To prove that it is non-degenerate, we apply Proposition \ref{prop: on hp 2}. Indeed   in this case
  \begin{equation*}
    \mathcal{B}
    = \begin{pmatrix}
      \mu_1 &\beta&\dots&\beta\\
      \beta &\mu_2&\dots&\beta\\
      \vdots&\vdots&\ddots&\vdots\\
      \beta &\beta&\dots&\mu_k\\
    \end{pmatrix},
  \end{equation*}
 whose  elements are strictly positive. It is also easy to check   that it is invertible if $\beta$ is large, since
  \begin{equation*}
    \det\mathcal{B}
    \sim\beta^k\det\begin{pmatrix}
      0&1&\dots&1\\
      1&0&\dots&1\\
      \vdots&\vdots&\ddots&\vdots\\
      1&1&\dots&0\\
    \end{pmatrix}
    = \beta^k(-1)^{k-1}(k-1)\quad \text{as}\ \beta\to+\infty.
  \end{equation*}
\end{proof}

\bibliography{normalized}
\bibliographystyle{abbrv}

\end{document}